\documentclass{amsart}
\usepackage{empheq}

\usepackage{fullpage,url,amssymb,enumerate,colonequals}

\usepackage{mathrsfs} 
\usepackage{MnSymbol}
\usepackage{extarrows}
\usepackage{lscape}
\usepackage[all,cmtip]{xy}

\usepackage[OT2,T1]{fontenc}

\usepackage{color}

\usepackage[
        colorlinks, citecolor=green,
        backref,
        pdfauthor={Imin Chen, Gleb Glebov, Ritesh Goenka},
]{hyperref}

\usepackage{float}
\restylefloat{table}

\usepackage[all]{xy}

\theoremstyle{definition}

\newtheorem{theorem}{Theorem}[section]

\newtheorem*{remark}{Remark}

\newtheorem{lemma}{Lemma}

\newcommand{\SL}{\operatorname{SL}}
\newcommand{\End}{\operatorname{End}}
\newcommand{\Z}{\mathbb{Z}}
\newcommand{\Q}{\mathbb{Q}}
\newcommand{\R}{\mathbb{R}}
\newcommand{\C}{\mathbb{C}}
\def\F21{\vphantom{\quad}_2 F_1}

\begin{document}

\title{Chudnovsky-Ramanujan type formulae for non-compact arithmetic triangle groups}

\author{Imin Chen ${}^{1}$, Gleb Glebov, and Ritesh Goenka ${}^{2}$}

\date{\today}



\address{Imin Chen \\
Department of Mathematics \\
Simon Fraser University \\
Burnaby \\
British Columbia \\
CANADA.}

\email{ichen@sfu.ca}

\address{Gleb Glebov \\
Department of Mathematics \\
Simon Fraser University \\
Burnaby \\
British Columbia \\
CANADA.}

\email{gglebov@sfu.ca}

\address{Ritesh Goenka \\
Department of Computer Science and Engineering \\
Indian Institute of Technology Bombay \\
Mumbai \\
Maharashtra \\
INDIA.}

\email{160050047@iitb.ac.in}

\thanks{${}^{1}$ is the corresponding author and was supported by an NSERC Discovery Grant and a SFU VPR Bridging Grant. ${}^{2}$ was supported by a Mitacs Globalink Award.}

\begin{abstract}
  We develop a uniform method to derive Chudnovsky-Ramanujan type formulae for triangle groups based on a generalization of a method of Chudnovsky and Chudnovsky; in particular, we carry out the method systematically for non-compact arithmetic triangle groups and one non-Fuchsian covering. As a result, we derive all rational Ramanujan type series given by Chan-Cooper for levels 1-4, as well as two additional rational series of a similar form prescribed by Chan-Cooper for these levels, but not found in the paper of Chan-Cooper. These two additional series were first found by Z.-W.\ Sun in a slightly different form. We also derive additional rational series of a similar form, but not found in the papers of Chan-Cooper nor Z.-W.\ Sun.
  
  As an ingredient in the method, we give an algorithm to rigorously confirm the singular values of normalized Eisenstein series of weight $2$, which may be of independent interest. 
\end{abstract}

\maketitle

\tableofcontents

\section{Introduction}

The transcendental constant $\pi$ has piqued human curiosity throughout history and many formulae and approximations have been given for it. A particularly fascinating family of formulae was first discovered by Ramanujan \cite[(28)--(44)]{Ramanujan} where he derived 17 series for $1/\pi$. Of note is the rapid convergence of some of these series and the fact that his derivations were rooted in the theory of modular functions, which has stimulated and played a central role in modern number theory. Later, Chudnovsky and Chudnovsky \cite{Chudnovsky}, \cite{Chudnovsky2} derived an additional such series based on the modular $j$-function, which is often used in practice for the computation of the digits of $\pi$ due to its rapid convergence. We recall their formula (in a slightly different form) below for later reference.
\begin{equation}
\label{clausen}
\frac{426880 \sqrt{10005}}{\pi} = \sum_{n = 0}^\infty \frac{(6n)!}{(3n)! n!^3} \bigg(545140134 n + 13591409 \bigg)  (-640320)^{-3n}.
\end{equation}
We note that the coefficients of the power series in \eqref{clausen} are a product of a coefficient of a generalized hypergeometric function and a linear function of the summation index. In addition, the value of the variable of the power series is rational.

The proofs of Ramanujan's series have traditionally relied on specialized knowledge of modular functions and their functional equations. For surveys of such results and methods, we mention \cite[Chapter 18]{Lost}, \cite{BBC}, \cite{Zudilin}. For example, in \cite{Borweins}, all 17 of Ramanujan's series were proven. See also \cite{BB1}, \cite{BB2} \cite[Chapter 14]{Cooper} for proofs of Ramanujan's identities and additional series. Finally, we mention the impressive work of Chan and Cooper \cite{Chan-Cooper}, where 186 series are derived for $1/\pi$ using methods similar in spirit to Ramanujan's original approach.

In \cite{Chen}, a generalization of \cite{Chudnovsky}, \cite{Chudnovsky2} was used to derive a complete list of Chudnovsky-Ramanujan type formulae for the modular $j$-function, even near elliptic points. The method gives an explanation of the different formulae for the modular $j$-function in terms of Kummer's 24 solutions to the hypergeometric differential equation and families of elliptic curves, and suggests applicability to any arithmetic triangle group of genus zero.

The method starts with a family of elliptic curves over $\mathbb{P}^1$ with three singular fibers, which is associated to a modular curve. The Picard-Fuchs differential equation is derived for this family and its solutions are expressed in terms of hypergeometric functions. This allows one to give expressions for the periods of the above family of elliptic curves in terms of hypergeometric functions, using Kummer's list of 24 solutions to the hypergeometric differential equation (cf.\ \cite{Dwork}). By the method in \cite{Chen} \cite{Chudnovsky} \cite{Chudnovsky2}, we obtain a precursor formula for $1/\pi$ which is valid in a certain simply connected domain of the fundamental domain near each elliptic or cusp point of the modular curve.

In order to obtain a Ramanujan type series such as \eqref{clausen} from a precursor formula, the hypergeometric parameters $a,b,c$ must satisfy $c = a + b + 1/2$ so that we may use Clausen's identity to simplify the precursor formula into a single summation series (henceforth, we refer to such a situation as a Clausen case). Associated to the parameters $a,b,c$ are the angles $\lambda \pi, \mu \pi, \nu \pi$ of the triangle group which is the monodromy group of the hypergeometric differential equation. The denominators of $\lambda, \mu, \nu$, respectively, give a triple $(p,q,r)$ is known as the signature of the triangle group.

The purpose of this paper is to show how the method in \cite{Chen} \cite{Chudnovsky} \cite{Chudnovsky2} can be further developed to derive a complete list of Chudnovsky-Ramanujan type formulae arising from modular curves corresponding to non-compact arithmetic triangle groups. The arithmetic triangle groups have been classified in both compact and non-compact cases in \cite{Takeuchi}. The non-compact cases were studied in \cite{Harnad-Mckay}, where more detailed information concerning their uniformizers is obtained and related to the replicable functions of Moonshine Theory.

We also provide a method to rigorously confirm singular values of $s_2(\tau)$ for $\tau$ an imaginary quadratic irrational. This fills a gap in the literature which allows us to directly deal with this obstruction. For example, the precursor formulae can be evaluated at singular moduli with higher class number to produce formulae such as in \cite{Borweins2}, with a uniform method to directly establish the values of $s_2(\tau)$ occurring in the precursor formula.

Of Ramanujan's 17 series for $1/\pi$, we are able to derive all formulae using the point of view suggested in \cite{Chen}, except those corresponding to signature $(2,4,\infty)$ and $(2,6,\infty)$. For signatures $(2,4,\infty)$ and $(2,6,\infty)$, the relevant arithmetic triangle subgroup is not contained in $\SL_2(\Z)$. However, the method can be modified by constructing a suitable relation between hypergeometric functions, and reducing to a previously derived precursor formula for a congruence subgroup. This method is suggested in \cite[p.46]{Chudnovsky2} and can also be applied to the previous signatures above. For signature $(2,4,\infty)$ and $(2,6,\infty)$, we carry out the method only for the Clausen cases. 

The signatures $(3,3,\infty)$, $(4,4,\infty)$, and $(6,6,\infty)$ do not give rise to Clausen cases; it should be possible to carry out the methods in this paper to find precursor formulae near the singular points for these cases, but we do not do so. For the compact arithmetic triangle groups listed in \cite{Takeuchi}, there are no further Clausen cases. Due to the lack of cusps, these cases are more difficult and less studied in the literature. In this direction, we point out the results in \cite{Yang}, which provide Ramanujan type identities arising from Shimura curves.

As a result, this paper derives in a uniform way all previously known rational Ramanujan type series for $1/\pi$ whose coefficients are linear in the summation index of the series. Our list of $36$ Clausen cases coincides exactly with series in \cite{Chan-Cooper} involving a single summation (see Section~\ref{uniqueness}). The methods of this paper give evidence that the list of single summation series for $1/\pi$ in \cite{Chan-Cooper} is complete in the sense that our uniform method can only produce a single summation formula in these $36$ cases when applied to any arithmetic triangle group.

In \cite{Chan-Cooper}, a rational series for $1/\pi$ is defined to be a series in rational numbers which converges to $C/\pi$ for some algebraic number $C$. Motivated by the methods used in this paper, we define a rational Ramanujan type series for $1/\pi$ to be a polynomial expression with rational coefficients in rational evaluations of generalized hypergeometric functions which equals $C/\pi$ for some algebraic number $C$. This definition for a rational Ramanujan type series suggests we can consider such series up to `equivalence' by hypergeometric identities, in order to distinguish `fundamentally' different formulae. 

A natural question is to explain using the methods in the paper the double summation series in \cite{Chan-Cooper}, and the series cited in \cite[Chapter 18]{Lost} and \cite{Zudilin}, which have coefficients in higher degree of the summation index. We answer the first question by showing the remaining rational double summation series for levels 1-4 in \cite{Chan-Cooper} arise from a different Kummer solution and the application of Euler's hypergeometric identity. As a result of the systematic application of the method, we also produce two additional rational double summation series of a similar form prescribed for levels 1-4 in \cite{Chan-Cooper}, but not found in \cite{Chan-Cooper}. These two additional series are not new but can be found in \cite{sun2}. In Table~\ref{tab:convolutional-double}, we also list some double summation series of a similar form as in \cite{sun2}, but not found in \cite{sun2}.

Many rational Ramanujan type series in the literature are expressed using binomial coefficients. It would be interesting to understand more precisely when a rational Ramanujan type series in our sense has a binomial coefficient form.

\section{Modular curves}

\label{modular-curves}

In this section, we summarize basic information about the non-compact arithmetic triangle groups $\Gamma$. There are five cases of $\Gamma$ contained in $\SL_2(\Z)$, labelled $1A$, $2a$, $2B$, $2C$, and $3B$ in \cite[Table 1]{Harnad-Mckay} (see also \cite{Cummins-Pauli}). They are all of genus zero and hence there exists a uniformizer $\xi$ which gives an isomorphism $\xi : \Gamma \backslash \mathfrak{H}^*$ to $\mathbb{P}^1(\C)$ which is unique up to a M\"obius transformation. Here, $\mathfrak{H}^* = \mathfrak{H} \cup \mathbb{P}^1(\Q)$. There are four cases of $\Gamma$ not contained in $\SL_2(\Z)$, labelled $2A, 3A, 4a, 6a$ in \cite[Table 1]{Harnad-Mckay}. Since the cases $2a, 4a$, and $6a$ do not contain any Clausen cases, we omit their discussion. 

We also consider the Riemann surface associated to the multi-valued function $\frac{1 + \sqrt{1-J(\tau)^{-1}}}{2}$, which we denote by $X_{1B}$, and covers $X(1) = \SL_2(\Z) \backslash \mathfrak{H}^*$, but does not arise as the quotient of $\mathfrak{H}^*$ by a Fuchsian group. This case is motivated by giving an explanation of the level 1 companion series in \cite{Chan-Cooper} in our framework. Consideration of this particular elliptic surface is suggested from \cite{almkvist}, for which an explicit determination is given in \cite{miranda}.

Throughout, let $J = J(\tau)$ denote the absolute Klein invariant, where $\tau \in \mathfrak{H}$.  Table~\ref{tab:groups} gives the signature, uniformizer and modular relation for each of the five triangle Fuchsian groups under consideration, and applicable information in the case of the non-Fuchsian case $1B$.

\begin{table}
\centering
\begin{tabular}{ccccc}
\hline \\[-7pt]
Name & Group & Signature & Uniformizer expression & Modular relation \\[3pt]
\hline \\[-7pt]
1B &  &  & $\dfrac{1 + \sqrt{1-J(\tau)^{-1}}}{2}$ & $J = \dfrac{1}{4J_{1B}(1-J_{1B})}$ \\[14pt]
2A & $\Gamma_0^+(2)$ & $(4,2,\infty)$ & $\dfrac{1}{16}
\left(\dfrac{\vartheta_3(\tau)^4 + \vartheta_4(\tau)^4}{\vartheta_2(\tau)^2 \vartheta_3(\tau) \vartheta_4(\tau)}\right)^4$ & $J_{2A} = -\dfrac{(1-J_{2B})^2}{4J_{2B}}$ \\[14pt]
2B & $\Gamma_0(2)$ & $(\infty,2,\infty)$ & $-\dfrac{1}{64}\left(\dfrac{\eta(\tau)}{\eta(2\tau)}\right)^{24}$ & $J = \dfrac{(4-J_{2B})^3}{27 J_{2B}^2}$ \\[14pt]
2C & $\Gamma(2)$ & $(\infty,\infty,\infty)$ & $\left(\dfrac{\vartheta_2(\tau)}{\vartheta_3(\tau)}\right)^4$ & $J = \dfrac{4}{27} \dfrac{(1 - J_{2C} (1 - J_{2C}))^3}{J_{2C}^2 (1 - J_{2C})^2}$ \\[14pt]
3A & $\Gamma_0^+(3)$ & $(6,2,\infty)$ & $\dfrac{(\eta(\tau)^{12} + 27 \eta(3\tau)^{12})^2}{108 \eta(\tau)^{12} \eta(3\tau)^{12}}$ & $J_{3A} = -\dfrac{(1-J_{3B})^2}{4J_{3B}}$ \\[14pt]
3B & $\Gamma_0(3)$ & $(\infty,3,\infty)$ & $-\dfrac{1}{27}\left(\dfrac{\eta(\tau)}{\eta(3\tau)}\right)^{12}$ & $J = \dfrac{(J_{3B}-9)^3 (1-J_{3B})}{64 J_{3B}^3}$ \\[12pt]
\hline
\end{tabular}
\caption{Details of coverings considered in this paper}
\label{tab:groups}
\end{table}

Here, $J_{\text{Name}}$ refers to the uniformizer, $\eta$ is the Dedekind eta function, and $\vartheta_2, \vartheta_3, \vartheta_4$ are the Jacobi theta functions. Note that $J_{2C}(\tau) = \lambda(\tau)$ for $\tau \in \mathfrak{H}$, where $\lambda$ is the modular lambda function. For ease of notation, we will henceforth denote the uniformizers $J_{1B}, J_{2A}, J_{2B}, J_{2C}, J_{3A}, J_{3B}$ by $s, v, t, \lambda, w, u$ respectively.

\begin{table}
\centering
\begin{tabular}{ccc}
\hline \\[-7pt]
Group & Special points & Stabilizer or cycle transformation \\[3pt]
\hline \\[-7pt]
& $0,\; \infty,\; \rho$ & $\begin{pmatrix} 1 & 0 \\ 1 & 1 \end{pmatrix}, \begin{pmatrix} 1 & 1 \\ 0 & 1 \end{pmatrix}, \begin{pmatrix} -1 & -1 \\ 1 & 0 \end{pmatrix}$ \\[14pt]
$\Gamma_0^+(2)$ & $\dfrac{-1+i}{2},\; \dfrac{\sqrt{-2}}{2},\; \infty$ & 
$\begin{pmatrix} -2 & -1 \\ 2 & 0 \end{pmatrix},\;
\begin{pmatrix} 0 & 1 \\ -2 & 0 \end{pmatrix},\;
\begin{pmatrix} 1 & 1 \\ 0 & 1 \end{pmatrix}$ \\[14pt]
$\Gamma_0(2)$ & $0,\; \dfrac{1+i}{2},\; \infty$ & $
\begin{pmatrix} 1 & 0 \\ -2 & 1 \end{pmatrix},\;
\begin{pmatrix} 1 & -1 \\ 2 & -1 \end{pmatrix},\;
\begin{pmatrix} 1 & 1 \\ 0 & 1 \end{pmatrix}$ \\[14pt]
$\Gamma(2)$ & $\infty,\; 0,\; 1$ & $
\begin{pmatrix}   1 & 2 \\ 0 & 1 \end{pmatrix},\;
\begin{pmatrix}   1 & 0 \\ 2 & 1 \end{pmatrix},\;
\begin{pmatrix}   1 & -2 \\ 2 & -3  \end{pmatrix}$ \\[14pt]
$\Gamma_0^+(3)$ & $\dfrac{-3+\sqrt{-3}}{6},\; \dfrac{\sqrt{-3}}{3},\; \infty$ & 
$\begin{pmatrix} -3 & -1 \\ 3 & 0 \end{pmatrix},\;
\begin{pmatrix} 0 & 1 \\ -3 & 0 \end{pmatrix},\;
\begin{pmatrix} 1 & 1 \\ 0 & 1 \end{pmatrix}$ \\[14pt]
$\Gamma_0(3)$ & $0,\; \dfrac{3+\sqrt{-3}}{6},\; \infty$ & $
\begin{pmatrix}   1 & 0 \\ 3 & 1 \end{pmatrix},\;
\begin{pmatrix}   1 & -1 \\ 3 & -2  \end{pmatrix},\;
\begin{pmatrix}   1 & 1 \\ 0 & 1 \end{pmatrix}$ \\[12pt]
\hline
\end{tabular}
\caption{Special points of coverings}
\label{tab:special}
\end{table}

By a special point, we mean a cusp or elliptic point of a Fuchsian group or representative vertex of a hyperbolic polygon. In Table~\ref{tab:special}, we give an ordered triple $(z_1, z_2, z_3)$ of special points and generator for its stabilizer or cycle transformation \cite[p.221]{maskit} for each case listed in Table~\ref{tab:groups}. Here, each triple is ordered so that the uniformizer assumes the values $0, 1, \infty$ at the points $z_1, z_2, z_3$ respectively.

\section{Families of elliptic curves}

Let $E$ be an elliptic curve over $\C$ given by $y^2 = 4 x^3 - g_2 x - g_3$ in (classical) Weierstrass form. 
The quantity $\Delta(E) = g_2^3 - 27 g_3^2 \not= 0$ is called the (normalized) discriminant of $E$. The $j$-invariant of $E$ is defined by
$$j(E) = 12^3 \frac{g_2^3}{\Delta(E)},$$
and its absolute Klein invariant $J(E)$ is defined by
$$J(E) = \frac{j(E)}{12^3}.$$

For $r \ne 0$, the map $\varphi_r : (x, y) \mapsto (r^2 x, r^3 y)$ gives an isomorphism from $E$ to the elliptic curve $E'$, where
$$E : y^2 = 4x^3 - g_2 x - g_3 \quad \text{and} \quad E' : y^2 = 4x^3 - g_2' x - g_3',$$
and
\begin{align*}
g_2' &= r^4 g_2, \\
g_3' &= r^6 g_3, \\
\Delta(E') &= r^{12} \Delta(E).
\end{align*}

By the uniformization theorem, $E(\mathbb{C}) = E_{\Lambda}(\mathbb{C})$ for some lattice $\Lambda \subset \mathbb{C}$. We then have $E'(\mathbb{C}) = E_{r^{-1} \Lambda}(\mathbb{C})$ and the following commutative diagram:
\begin{displaymath}
\xymatrix{
\mathbb{C} / \Lambda
\ar[d]_-{z \mapsto r^{-1} z}
\ar[r]_-{\iota_\Lambda}
& E_\Lambda(\mathbb{C}) = E(\mathbb{C})
\ar[d]^-{\varphi_{r}} \\
\mathbb{C} / {r^{-1} \Lambda}
\ar[r]_-{\iota_{r^{-1} \Lambda}}
& E_{r^{-1} \Lambda}(\mathbb{C}) = E'(\mathbb{C})
}
\end{displaymath}
so that the isomorphism $\varphi_r$ corresponds to scaling $\Lambda$ by $r^{-1}$.

We now introduce three families of elliptic curves: $E_\tau, \widetilde{E}, E_J,$ and compare their discriminants, associated lattices, and periods. Consider the elliptic curve $E_\tau$ over $\C$ given by
$$E_\tau : y^2 = 4x^3 - g_2(\tau) x - g_3(\tau), \quad \Delta(E_\tau) = \Delta(\tau) = g_2(\tau)^3 - 27 g_3(\tau)^2, \quad \Lambda(E_\tau) = \mathbb{Z} + \mathbb{Z} \tau,$$
with discriminant $\Delta(E_\tau)$ and associated lattice $\Lambda(E_\tau)$. Let the $j$-invariant $j(E_\tau)$ and absolute Klein invariant $J(E_\tau)$ of $E_\tau$ be denoted by $j = j(\tau)$ and $J = J(\tau)$, respectively.

Taking $r = \Delta(\tau)^{-1/12}$, we see that $E_\tau$ is isomorphic to
$$\widetilde{E} : y^2 = 4x^3 - \gamma_2 x - \gamma_3, \quad \Delta(\widetilde{E}) = 1, \quad \Lambda(\widetilde{E}) = \mathbb{Z} \widetilde{\omega}_1 + \mathbb{Z} \widetilde{\omega}_2,$$
with discriminant $\Delta(\widetilde{E})$ and associated lattice $\Lambda(\widetilde{E})$, where
$$\gamma_2 = \Delta(\tau)^{-1/3} g_2(\tau) = J^{1/3}, \quad \gamma_3 = \Delta(\tau)^{-1/2} g_3(\tau) = \left(\frac{J - 1}{27}\right)^{1/2},$$
and
$$\widetilde{\omega}_1 = \Delta(\tau)^{1/12}, \quad \widetilde{\omega}_2 = \tau \Delta(\tau)^{1/12}.$$

Further, taking $r = \left(g_2(\tau)/g_3(\tau)\right)^{1/2}$, we see that $E_\tau$ is isomorphic to
$$E_J : y^2 = 4x^3 - gx - g, \quad \Delta(E_J) = \frac{3^9 J^2}{16 (1-J)^2}, \quad \Lambda(E_J) = \mathbb{Z} \Omega_1 + \mathbb{Z} \Omega_2,$$
with discriminant $\Delta(E_J)$ and associated lattice $\Lambda(E_J)$, where
$$g = \frac{g_2(\tau)^3}{g_3(\tau)^2} = \frac{27J}{J - 1},$$
and
$$\Omega_1 = \left(\frac{g_3(\tau)}{g_2(\tau)}\right)^{1/2}, \quad \Omega_2 = \tau \left(\frac{g_3(\tau)}{g_2(\tau)}\right)^{1/2}.$$

It is natural to study $E_J$ with $J$ given by the corresponding modular relation. For each of the four cases under consideration, we choose a value of $r = r_\xi^{1/2}$, so that $E_J$ is isomorphic to
$$E_\xi : y^2 = 4x^3 - g_2(\xi) x - g_3(\xi), \quad \Delta(E_\xi) = g_2(\xi)^3 - 27 g_3(\xi)^2, \quad \Lambda(E_\xi) = \mathbb{Z} \omega_1(\xi) + \mathbb{Z} \omega_2(\xi),$$
with discriminant $\Delta(E_\xi)$ and associated lattice $\Lambda(E_\xi)$, where
$$\omega_1(\xi) = \Delta(E_\xi)^{-1/12} \Delta(\tau)^{1/12}, \quad \omega_2(\xi) = \tau \Delta(E_\xi)^{-1/12} \Delta(\tau)^{1/12}.$$
Table~\ref{tab:elliptic} gives the chosen value of $r_\xi$ and the resulting values of $g_2(\xi)$, $g_3(\xi)$ and $\Delta(E_\xi)$ in each case.

\begin{table}[H]
\centering
\begin{tabular}{ccccc}
\hline \\[-7pt]
Case & $r_\xi$ & $g_2(\xi)$ & $g_3(\xi)$ & $\Delta(E_\xi)$ \\[3pt]
\hline \\[-7pt]
1B & $2s-1$ & $27$ & $27 (2s-1)$ & $2^2 3^9 s (1-s)$ \\[14pt]
2B & $\dfrac{(t - 1)(t + 8)}{(t - 4)}$ & $27 (t-1) (t-4)$ & $27 (t-1)^2 (t+8)$ & $-3^{12} t^2 (t-1)^3$ \\[14pt]
2C & $\dfrac{(\lambda + 1)(2\lambda - 1)(\lambda - 2)}{9(\lambda^2 - \lambda + 1)}$ & $\dfrac{4}{3} (\lambda^2 - \lambda + 1)$ & $\dfrac{4}{27} (\lambda + 1)(2\lambda - 1)(\lambda - 2)$ & $2^4 \lambda^2 (1 - \lambda)^2$ \\[14pt]
3B & $\dfrac{u^2+18u-27}{u-9}$ & $27 (u-1) (u-9)$ & $27 (u-1) (u^2+18u-27)$ & $- 2^6 3^9 u^3 (u-1)^2$ \\[10pt]
\hline
\end{tabular}
\caption{Parameters of the elliptic curve $E_\xi$}
\label{tab:elliptic}
\end{table}

\begin{remark}
Note that under the change of variables
\begin{equation*}
x \mapsto x - \frac{\lambda + 1}{3}, \quad y \mapsto 2y,
\end{equation*}
the family $E_\lambda$ becomes the classical Legendre family:
\begin{equation*}
C_\lambda : y^2 = x(x - 1)(x - \lambda).
\end{equation*}
This defines an isomorphism $\Phi : E_\lambda \to C_\lambda$.
\end{remark}

In the previous cases, which correspond to coverings of $X(1)$, we were able to write down a suitable family of elliptic curves $E_\xi$ over $\C(\xi)$; this must be modified to deal with the remaining cases.

Let $\Gamma$ be a triangular arithmetic group (i.e.\ cases $2B, 2C, 3B, 2A, 3A$). Recall that a uniformizer $\xi$ for $\Gamma$ is an isomorphism $\xi : \Gamma \backslash \mathfrak{H}^* \rightarrow \mathbb{P}^1(\C)$. Let $\tau = \tau (\xi)$ be a local inverse to the uniformizer $\xi$. Given a simply-connected fundamental domain $\mathcal{F} \subseteq \mathfrak{H}^*$ for $\Gamma$, we may extend $\tau$ to a local isomorphism $\tau : \mathbb{P}^1(\C) - \xi(\partial \mathcal{F}) \rightarrow  \mathcal{F}$. In the case $1B$, there is no associated Fuchsian group $\Gamma$, but we have an isomorphism $\xi : X_{1B} \rightarrow \mathbb{P}^1(\C)$ of Riemann surfaces.

In cases 2A and 3A, we introduce a new family of elliptic curves $\overline{E}_{\xi}$ over $\C(\frac{d\xi}{d\tau},g_2(\tau),g_3(\tau))$. Taking $r = \sqrt{\frac{d\xi}{d\tau}}$ and using the fact that $\tau = \tau(\xi)$ is a local isomorphism as above, we see that $E_\tau = E_{\tau(\xi)}$ is isomorphic to
$$\overline{E}_{\xi} : y^2 = 4x^3 - \overline{g}_2(\xi) x - \overline{g}_3(\xi), \quad \Delta(\overline{E}_{\xi}) = \left( \frac{d\xi}{d\tau} \right)^6 \Delta(\tau), \quad \Lambda(\overline{E}_{\xi}) = \mathbb{Z} \overline{\omega}_1(\xi) + \mathbb{Z} \overline{\omega}_2(\xi),$$
with discriminant $\Delta(\overline{E}_{\xi})$ and associated lattice $\Lambda(\overline{E}_{\xi})$, where
$$\overline{g}_2(\xi) = \left( \frac{d\xi}{d\tau} \right)^2 g_2(\tau), \quad \overline{g}_3(\xi) = \left( \frac{d\xi}{d\tau} \right)^3 g_3(\tau),$$
and
$$\overline{\omega}_1(\xi) = \sqrt{\frac{d\xi}{d\tau}}, \quad \overline{\omega}_2(\xi) = \tau \sqrt{\frac{d\xi}{d\tau}}.$$

To summarize the discussion on the families of elliptic curves described, we provide the following commutative diagram which depicts the isomorphisms between these families of elliptic curves:
\begin{displaymath}
\xymatrixcolsep{5pc}
\xymatrixrowsep{4pc}
\xymatrix{
\overline{E}_\xi(\mathbb{C}) & E_\tau(\mathbb{C})
\ar[d]|{\varphi_{\Delta(\tau)^{-1/12}}}
\ar[r]|{\varphi_{\left(\frac{g_2(\tau)}{g_3(\tau)}\right)^{1/2}}}
\ar[l]|{\varphi_{\left(\frac{d\xi}{d\tau}\right)^{1/2}}}
& E_J(\mathbb{C})
\ar[dl]|<<<<<<<<<{\varphi_{\Delta(E_J)^{-1/12}}}
\ar[d]|{\varphi_{r_\xi^{1/2}}} \\
& \widetilde{E}(\mathbb{C})
\ar[r]|{\varphi_{\Delta(E_\xi)^{1/12}}}
\ar[ul]|{\varphi_{\Delta(\overline{E}_\xi)^{1/12}}}
& E_\xi(\mathbb{C})
\ar[lu]|<<<<<<<<{\varphi_{\omega_1}}
}
\end{displaymath}

\section{The Picard-Fuchs differential equation}

We now prove a theorem that will allow us to systematically derive Picard-Fuchs differential equations for elliptic curves in Weierstrass form. We begin with the following elementary result:

\begin{lemma}
\label{systme-de}
Any system of two first order differential equations
\begin{empheq}[left=\empheqlbrace]{align}
\nonumber
\begin{split}
X' &= aX + bY, \\
Y' &= cX + dY,
\end{split}
\end{empheq}
where $X$, $Y$, and $a$, $b$, $c$, $d$ are functions of $x$, can be written as a second order differential equation:
$$X'' = \bigg\{a + d + \frac{b'}{b}\bigg\}X' + \bigg\{bc - ad + a' - \frac{ab'}{b}\bigg\}X.$$
\end{lemma}

\begin{proof}
Differentiating the first equation and using the second, we get
$$X'' = aX' + a'X + bY' + b'Y = aX' + a'X + b(cX + dY) + b'Y,$$
but from the first equation
$$Y = \frac{X' - aX}{b},$$
provided $b \neq 0$, so
\begin{align*}
X'' &= aX' + a'X + b\bigg\{cX + d\frac{X' - aX}{b}\bigg\} + b'\frac{X' - aX}{b} \\
&= aX' + a'X + bcX + d(X' - aX) + b'\frac{X' - aX}{b} \\
&= aX' + a'X + bcX + dX' - adX + \frac{b'}{b}X' - \frac{ab'}{b}X.
\end{align*}
Collecting like terms yields
$$X'' = \bigg\{a + d + \frac{b'}{b}\bigg\}X' + \bigg\{bc - ad + a' - \frac{ab'}{b}\bigg\}X.$$
\end{proof}

\begin{lemma}
Suppose $A$, $B$ are analytic functions of a parameter $\xi$ and are the invariants of the elliptic curve $E$ over $\mathbb{C}$ given by
$$E : y^2 = 4x^3 - A x - B, \quad A^3 - 27B^2 \ne 0.$$
Let $\gamma = \gamma_i = \gamma_i(\xi) \in H_1(E,\Z)$ be elements which are continuous functions of $\xi$, and let 
\begin{align*}
  p & = p_i = 2 \int_{\gamma_i} \frac{dx}{y} \\
  q & = q_i = 2 \int_{\gamma_i} \frac{x}{y} dx
\end{align*}
 be the associated periods and quasi-periods, respectively, where $k = 1, 2$. Let
\begin{empheq}[left=\empheqlbrace]{align}
\nonumber
\begin{split}
P &= \frac{-36B A' + 24A B'}{8(A^3 - 27B^2)}, \\
Q &= \frac{2A^2 A' - 36B B'}{8(A^3 - 27B^2)}, \\
R &= \frac{-3A B A' + 2A^2 B'}{8(A^3 - 27B^2)},
\end{split}
\end{empheq}
where
$$A' = \frac{dA}{d\xi}, \quad B' = \frac{dB}{d\xi},$$
then
$$\frac{d^2p}{d\xi^2} - \frac{P'}{P} \frac{dp}{d\xi} - \bigg\{Q^2 - PR - Q' + \frac{QP'}{P}\bigg\}p = 0.$$
\end{lemma}

\begin{proof}
Write
\begin{equation}
  \label{period-quasi-period}
  p = 2\int_\gamma \frac{dx}{y}, \quad q = -2\int_\gamma \frac{x}{y} dx,
\end{equation}
where $\gamma$ is a loop around two and only two roots of $y^2 = 4 x^3 - A x - B = 0$. Then
$$\frac{dp}{d\xi} = A' \int_\gamma \frac{x}{y^3} dx + B' \int_\gamma \frac{dx}{y^3}, \quad \frac{dq}{d\xi} = -A' \int_\gamma \frac{x^2}{y^3} dx - B' \int_\gamma \frac{x}{y^3} dx.$$
For brevity define
$$I_0 = \int_\gamma \frac{dx}{y^3}, \quad I_1 = \int_\gamma \frac{x}{y^3} dx, \quad I_2 \int_\gamma \frac{x^2}{y^3} dx,$$
so that
$$\frac{dp}{d\xi} = A' I_1 + B' I_0, \quad \frac{dq}{d\xi} = -A' I_2 - B' I_1.$$
Now, by the fundamental theorem of complex integration
$$\int_\gamma d(y^{-1}) = -12\int_\gamma \frac{x^2}{2y^3} dx + A \int_\gamma \frac{dx}{2y^3} = 0.$$
Likewise,
\begin{align*}
\int_\gamma d(xy^{-1}) &= \int_\gamma \frac{dx}{y} - 12\int_\gamma \frac{x^3}{2y^3} dx + A \int_\gamma \frac{x}{2y^3} \\
&= 2\int_\gamma \frac{4x^3 - A x - B}{2y^3} dx - 12\int_\gamma \frac{x^3}{2y^3} dx + A \int_\gamma \frac{x}{2y^3} dx \\
&= -4\int_\gamma \frac{x^3}{2y^3} dx - A\int_\gamma \frac{x}{2y^3} dx - 2B \int_\gamma \frac{dx}{2y^3} \\
&= 0.
\end{align*}
Moreover,
\begin{align*}
\int_\gamma d(x^2 y^{-1}) &= 2\int_\gamma \frac{x}{y} dx - 12\int_\gamma \frac{x^4}{2y^3} dx + A \int_\gamma \frac{x^2}{2y^3} \\
&= 4\int_\gamma x\frac{4x^3 - A x - B}{2y^3} dx - 12\int_\gamma \frac{x^4}{2y^3} dx + A \int_\gamma \frac{x^2}{2y^3} dx \\
&= 4\int_\gamma \frac{x^4}{2y^3} dx - 3A\int_\gamma \frac{x^2}{2y^3} dx - 4B \int_\gamma \frac{x}{2y^3} dx \\
&= 0.
\end{align*}
We thus arrive at a system of three equations:
\begin{align*}
0 &= 12I_2 - A I_0, \\
-\frac{p}{2} &= 2A I_1 + 3B I_0, \\
-\frac{q}{2} &= 2A I_2 + 3B I_1.
\end{align*}
Solving this system of equations leads to
\begin{align*}
I_0 &= \frac{3(3B p - 2 A q)}{2(A^3 - 27B^2)}, \\
I_1 &= \frac{18B q - A^2 p}{4(A^3 - 27B^2)}, \\
I_2 &= \frac{A (3B p - 2 A q)}{8(A^3 - 27B^2)},
\end{align*}
and consequently
\begin{empheq}[left=\empheqlbrace]{align}
\label{eq:Bruns}
\begin{split}
\frac{dp}{d\xi} &= - Qp - Pq, \\
\frac{dq}{d\xi} &= Rp + Qq.
\end{split}
\end{empheq}
Applying Lemma~\ref{systme-de} to \eqref{eq:Bruns} yields
$$\frac{d^2p}{d\xi^2} - \frac{P'}{P} \frac{dp}{d\xi} - \bigg\{Q^2 - PR - Q' + \frac{QP'}{P}\bigg\}p = 0.$$
\end{proof}

\begin{theorem}
\label{thm:Picard}
Suppose $\gamma_2$, $\gamma_3$ are analytic functions of a parameter $\xi$ and are the invariants of the elliptic curve $\widetilde{E}$ over $\mathbb{C}$ given by
$$\widetilde{E} : y^2 = 4x^3 - \gamma_2 x - \gamma_3, \quad \gamma_2^3 - 27\gamma_3^2 = 1,$$
with periods $\widetilde{\omega}_1$, $\widetilde{\omega}_2$. If $\widetilde{\omega} = \widetilde{\omega}_k$ with $k = 1, 2$, then
$$\frac{d^2 \widetilde{\omega}}{d\xi^2} + \bigg\{\frac{7J - 4}{6(J - 1)} \frac{J'}{J} - \frac{J''}{J'}\bigg\} \frac{d\widetilde{\omega}}{d\xi} + \bigg\{\frac{J'^2}{144J(J - 1)}\bigg\} \widetilde{\omega} = 0.$$
\end{theorem}

\begin{proof}
Note that
$$Q^2 - PR = \frac{(A^3 - 27B^2)(AA'^2 - 12B'^2)}{16}.$$
Thus
$$\frac{d^2p}{d\xi^2} - \frac{X'}{X} \frac{dp}{d\xi} - \bigg\{Z - Y' + \frac{YX'}{X}\bigg\} \frac{p}{16} = 0,$$
where
\begin{empheq}[left=\empheqlbrace]{align}
\nonumber
\begin{split}
X &= -72\gamma_3 \gamma_2' + 48\gamma_2 \gamma_3', \\
Y &= 4\gamma_2^2 \gamma_2' - 72\gamma_3 \gamma_3', \\
Z &= \gamma_2 \gamma_2'^2 - 12\gamma_3'^2.
\end{split}
\end{empheq}
However, a labourious calculation shows that
\begin{empheq}[left=\empheqlbrace]{align}
\nonumber
\begin{split}
X &= \frac{8}{J^{2/3} \sqrt{3(J - 1)}} J', \\
Y &= 0, \\
Z &= \frac{J'^2}{9J(1 - J)}.
\end{split}
\end{empheq}
Therefore,
$$\frac{X'}{X} = \frac{J''}{J'} + \frac{(7J - 4)}{6(1 - J)} \frac{J'}{J},$$
and thus
$$\frac{d^2 \widetilde{\omega}}{d\xi^2} + \bigg\{\frac{7J - 4}{6(J - 1)} \frac{J'}{J} - \frac{J''}{J'}\bigg\} \frac{d\widetilde{\omega}}{d\xi} + \bigg\{\frac{J'^2}{144J(J - 1)}\bigg\} \widetilde{\omega} = 0.$$
\end{proof}

\begin{theorem}
\label{thm:Picard-2}
Suppose $\overline{g}_2(\xi)$, $\overline{g}_3(\xi)$ are analytic functions of a parameter $\xi$ and are the invariants of the elliptic curve $\overline{E}(\xi)$ over $\mathbb{C}$ given by
$$\overline{E}(\xi) : y^2 = 4x^3 - \overline{g}_2(\xi) x - \overline{g}_3(\xi),$$
with periods $\overline{\omega}_1(\xi)$, $\overline{\omega}_2(\xi)$. If $\overline{\omega}(\xi) = \overline{\omega}_k(\xi)$ with $k = 1, 2$, then
$$\frac{d^2}{d\xi^2}\overline{\omega}(\xi) + R(\xi) \overline{\omega}(\xi) = 0,$$
where $R(\xi)$ is an algebraic function of $\xi$. Moreover, $R(\xi)$ in the above equation is given by
\begin{equation*}
    R(\xi) = \frac{1}{4} \left( \frac{1 - \lambda^2}{\xi^2} + \frac{1 - \mu^2}{(\xi - 1)^2} + \frac{\lambda^2 + \mu^2 - \nu^2 - 1}{\xi (\xi - 1)}
    \right),
\end{equation*}
where $\lambda, \mu, \nu$ are the exponents at the regular singular points $0, 1, \infty$, respectively, which determine the angles $\lambda \pi, \mu \pi, \nu \pi$ at the respective vertices.
\end{theorem}

\begin{proof}
The first part of the statement follows directly from the proof of \cite[Theorem 15, pp. 99-101]{Ford} and the expressions for $\overline{g}_2(\xi)$, $\overline{g}_3(\xi)$. The expression for $R(\xi)$ is given in \cite[p. 7]{Harnad-Mckay}.
\end{proof}

\section{Solutions to the Picard-Fuchs differential equation}

\label{picard-fuchs-solution}

Applying Theorem~\ref{thm:Picard} or Theorem~\ref{thm:Picard-2} will often lead to differential equations with three singularities, which we now discuss in detail. As a general reference for the material in this section and hypergeometric functions, see \cite{Andrews}.

Let
\begin{equation}
\label{eqn:Riemann}
    \frac{d^2 u}{dz^2} + p(z) \frac{du}{dz} + q(z) u = 0
\end{equation}
have three, and only three regular singularities, $z_1$, $z_2$, and $z_3$, with respective exponents $\alpha, \alpha'$; $\beta, \beta'$; and $\gamma, \gamma'$, satisfying
$$\alpha + \alpha' + \beta + \beta' + \gamma + \gamma' = 1.$$
To express the fact that $u$ satisfies an equation of this type, Riemann wrote
$$u = P\left\{
\begin{matrix}
z_1 & z_2 & z_3 & \; \\
\alpha & \beta & \gamma & z \\
\alpha' & \beta' & \gamma' & \;
\end{matrix}
\right\}.
$$
A differential equation of this type is called Riemann's $P$-equation. 

The hypergeometric equation
\begin{equation}
\label{eqn:hypergeometric}
z(1 - z) \frac{d^2 u}{dz^2} + [c - (a + b + 1)z] \frac{du}{dz} - abu = 0
\end{equation}
has three regular singular points $0$, $1$, $\infty$ and is defined by the Riemann scheme 
$$P\left\{
\begin{matrix}
0 & 1 & \infty & \; \\
0 & 0 & a & z \\
1 - c & c - a - b & b & \;
\end{matrix}
\right\}.
$$
To see this, note that if $z_0 = 0$, then the indicial equation is
$$r(r - 1) + cr = 0$$
and its roots are $r = 0, 1 - c$; if $z_0 = 1$, then the indicial equation is
$$r(r - 1) + (1 + a + b - c)r = 0$$
and its roots are $r = 0, c - a - b$; if $z_ 0 = \infty$, then the indicial equation is
$$r(r - 1) + (1 - a - b)r + ab = 0$$
and its roots are $r = a, b$.

The classical hypergeometric function is given by
\begin{equation*}
\F21(a,b;c;z) = \sum_{n=0}^\infty \frac{(a)_n (b)_n}{(c)_n} \frac{z^n}{n!},
\end{equation*}
where
\begin{equation*}
    (\alpha)_n = \alpha (\alpha+1) \cdots (\alpha + n - 1)
\end{equation*}
is the Pochhammer symbol.

There are 24 possible hypergeometric solutions to the hypergeometric differential equation and they are known as Kummer's solutions \cite{Kummer}. In general, there are 8 solutions around each singular point; each of these breaks up into two equivalence classes under Pfaff and Euler transformations. We denote the set of Kummer's solutions to \eqref{eqn:hypergeometric} by $K(a, b, c)$.

If the differential equation \eqref{eqn:Riemann} is defined by the Riemann scheme
$$P\left\{
\begin{matrix}
0 & 1 & \infty & \; \\
\alpha & \beta & \gamma & z \\
\alpha' & \beta' & \gamma' & \;
\end{matrix}
\right\},
$$
then from \cite[Section 10.7]{Watson}, we have
$$u = P\left\{
\begin{matrix}
0 & 1 & \infty & \; \\
\alpha & \beta & \gamma & z \\
\alpha' & \beta' & \gamma' & \;
\end{matrix}
\right\}
= z^\alpha (1 - z)^{\beta}
P\left\{
\begin{matrix}
0 & 1 & \infty & \; \\
0 & 0 & \alpha + \beta + \gamma & z \\
\alpha' - \alpha & \beta' - \beta & \alpha + \beta + \gamma' & \;
\end{matrix}
\right\}.
$$
In other words, if $F(z)$ is one of Kummer's solutions to the hypergeometric differential equation \eqref{eqn:hypergeometric} with
$$a = \alpha + \beta + \gamma, \quad b = \alpha + \beta + \gamma', \quad c = 1 + \alpha - \alpha',$$
then
$$u = z^{\alpha} (1 - z)^{\beta} F(z).$$

For the first four cases, consider the elliptic curve $\widetilde{E}$ parametrized by the uniformizer using the modular relation. From Theorem~\ref{thm:Picard}, it is possible to determine the Picard Fuchs differential equation satisfied by the periods of $\widetilde{E}$. Following the method outlined above, we derive hypergeometric representations of $\widetilde{\omega}$ of the form
\begin{equation*}
    \widetilde{\omega} = \xi^{\alpha} (1 - \xi)^{\beta} F(\xi),
\end{equation*}
where $F$ belongs to the set of Kummer's solutions $K(a,b,c)$. In Table~\ref{tab:period-1}, we list these solution parameters for the four cases under consideration.

\begin{table}[H]
\centering
\begin{tabular}{cccccc}
\hline \\[-7pt]
Case & $\alpha$ & $\beta$ & $a$ & $b$ & $c$ \\[3pt]
\hline \\[-7pt]
1B & 1/12 & 1/12 & 1/6 & 5/6 & 1 \\[9pt]
2B & 1/6 & 0 & 1/4 & 1/4 & 1 \\[9pt]
2C & 1/6 & 1/6 & 1/2 & 1/2 & 1 \\[9pt]
3B & 1/4 & 0 & 1/3 & 1/3 & 1 \\[3pt]
\hline
\end{tabular}
\caption{Solution parameters for the period $\widetilde{\omega}$}
\label{tab:period-1}
\end{table}

In the previous section, we introduced another family of elliptic curves $\overline{E}_\xi$ to tackle the remaining cases. From Theorem~\ref{thm:Picard-2}, it is possible to determine the Picard Fuchs differential equation satisfied by the periods of $\overline{E}_\xi$. Following the method outlined above, we derive hypergeometric representations of $\overline{\omega}(\xi)$ of the form
\begin{equation*}
    \overline{\omega}(\xi) = \xi^{\alpha} (1 - \xi)^{\beta} F(\xi),
\end{equation*}
where $F$ belongs to the set of Kummer's solutions $K(a,b,c)$. In Table~\ref{tab:period-2}, we list these solution parameters for the cases 2A, 2B, 3A and 3B.

\begin{table}[H]
\centering
\begin{tabular}{cccccc}
\hline \\[-7pt]
Case & $\alpha$ & $\beta$ & $a$ & $b$ & $c$ \\[3pt]
\hline \\[-7pt]
2A & 3/8 & 1/4 & 1/8 & 1/8 & 3/4 \\[9pt]
2B & 1/2 & 1/4 & 1/4 & 1/4 & 1 \\[9pt]
3A & 5/12 & 1/4 & 1/6 & 1/6 & 5/6 \\[9pt]
3B & 1/2 & 1/3 & 1/3 & 1/3 & 1 \\[3pt]
\hline
\end{tabular}
\caption{Solution parameters for the period $\overline{\omega}(\xi)$}
\label{tab:period-2}
\end{table}

\section{Hypergeometric representations of periods}

From Section~\ref{picard-fuchs-solution}, we know that near a singular point,
$\widetilde{\omega}$ is expressible in terms of hypergeometric functions. In this section, we explicitly calculate this period expression with known constants up to a 12th root of unity. 

The period expression is valid in a certain simply-connected domain, and by evaluation numerically, one can compute the 12th root of unity precisely. However, for simplicity, we will delay specifying the exact root of unity until the final examples where we specialize $\tau$ to specific values. Also, for brevity, we only exhibit period expressions near the cusp $\tau = \infty$.

For the cases 1B, 2B, 2C, 3B, we obtain a hypergeometric representation of the first period of elliptic curve $E_\xi$. If $\widetilde{\omega}$ is a solution of the Picard-Fuchs differential equation for $\widetilde{E}$, then there are constants $A$ and $B$ such that
$$\widetilde{\omega} = A\widetilde{\omega}_1 + B\widetilde{\omega}_2 = (A + B\tau) \Delta(\tau)^{1/12}.$$ Equating this with a hypergeometric solution locally near a special point, we can solve for the constants $A,B$ by using functional identities under the stabilizers of the special point. In this section, we use the convention that $\tau \rightarrow \infty$ means $\text{Im}(\tau) \rightarrow \infty$.

Let $\mathcal{F}$ be a fundamental domain for the covering associated to the cases $1B, 2B, 2C, 3B, 2A, 3A$ and $\xi$ be the uniformizer specified in Section~\ref{modular-curves}. Let $\nu(\xi)$ be one of the following six expressions
$$\xi, 1/\xi, 1-\xi, 1/(1-\xi), (\xi-1)/\xi, \xi/(\xi-1),$$
and
$$C_{\nu(\xi)} = \left\{ \tau \in \mathcal{F} : |\nu(\xi)| < 1 \right\}.$$
By the argument in \cite[Lemma 3.1]{Chen}, $C_{\nu(\xi)}$ is the union of at most two connected components, each of which is simply-connected. As a result, we can define uniquely the periods of the elliptic curve $E_\xi$ on one of these simply-connected components using the discussion in \cite[\S 3.2]{Chen}. For a special point $\tau$ lying in the closure of $\mathcal{F}$, we choose a connected component $C_{\nu(\xi),\tau}$ of $C_{\nu(\xi)}$ which contains $\tau$ in its closure.

The covers to which we apply the path lifting lemma as in \cite[\S 3.2]{Chen} to uniquely define the periods are listed below:
\begin{enumerate}
    \item Case 1B: $X(2) \times_{X(1)} X_{1B} \rightarrow X_{1B}$
    \item Case 2B: $X(2) \rightarrow X_0(2)$
    \item Case 2C: $X(2) \rightarrow X(2)$
    \item Case 3B: $X(2) \times_{X(1)} X_0(3) \rightarrow X_0(3)$
    \item Case 2A: $X(2) \rightarrow X_0(2) \rightarrow X_0^+(2)$
    \item Case 3A: $X(2) \times_{X(1)} X_0(3) \rightarrow X_0(3) \rightarrow X_0^+(3)$,
\end{enumerate}
where $X \times_Z Y$ denotes the fiber product of $X$ and $Y$ over $Z$. Here, $X(N) = \Gamma(N) \backslash \mathfrak{H}^*$ and $X_0(N) = \Gamma_0(N) \backslash \mathfrak{H}^*$ are the classically denoted modular curves.

For the cases $1B, 2B, 2C, 3B$, we use the family $E_\xi$. For the cases $2A, 3A$, we use the family $\overline{E}_\xi$, keeping in mind that $\tau = \tau(\xi)$ is a local inverse to $\xi$.

\subsection{Case 1B}

The fundamental domain we take in this case is
\begin{equation*}
    \mathcal{F}_{1B} = \left\{ z \in \mathfrak{H} : |z+1| > 1, |z-1| > 1, |{\rm Re}(z)| < 1/2 \right\},
\end{equation*}
which can be obtained by pulling back the branch cut for the principal value of $\sqrt{1-J(\tau)^{-1}}$ in the multi-valued function
\begin{equation*}
  \frac{1 + \sqrt{1-J(\tau)^{-1}}}{2},
\end{equation*}
to the closure of the fundamental domain for $J(\tau)$, and using numerical values to identify sheets as we cross the branch cuts. As noted in Section~\ref{modular-curves}, $\mathcal{F}_{1B}$ does not arise as the fundamental domain of a Fuchsian group acting on $\mathfrak{H}^*$ as it does not satisfy the necessary cycle condition \cite[Theorem 9.3.5.]{beardon}.

The transformation $\tau \mapsto \frac{\tau}{\tau+1}$ stabilizes the special point at $\tau = 0$ and identifies the two arc edges of $\mathcal{F}_{1B}$. Hence
\begin{align*}
  (A + B\tau) \Delta(\tau)^{1/12} & = \left(A + B\left(\frac{\tau}{\tau+1}\right)\right) \Delta\left(\frac{\tau}{\tau+1}\right)^{1/12} \\
  & = \left(A + B\left(\frac{\tau}{\tau+1}\right)\right) e^{-\pi i/6} (\tau+1) \Delta(\tau)^{1/12},
\end{align*}
that is, around $\tau = 0$,
\begin{equation*}
  (A + B\tau) = (A(\tau+1) + B\tau) e^{-\pi i/6}.
\end{equation*}
Letting $\tau \to 0$, we obtain $A = 0$. Since around $s = 0$, one of the solutions (corresponding to $m=1, 2, 17, 18$ in \cite[Table I]{Dwork}) is
\begin{equation*}
  \widetilde{\omega} = s^{1/12} (1-s)^{1/12} \vphantom{1em}_2 F_1 \left( \frac{1}{6}, \frac{5}{6}; 1; s \right),
\end{equation*}
we see that around $\tau = 0$,
\begin{equation}
\label{delta-expr:s=0}
  s^{1/12} (1-s)^{1/12} \vphantom{1em}_2 F_1 \left( \frac{1}{6}, \frac{5}{6}; 1; s \right) = B \tau \Delta(\tau)^{1/12}.
\end{equation}
Using the modular relation for this case, we obtain
\begin{align*}
    B &= \lim_{\tau \rightarrow 0} \frac{1}{2^{1/6} \tau \Delta(\tau)^{1/12} J^{1/12}} \\
    &= \lim_{\tau \rightarrow 0} \frac{1}{2^{1/6} \tau g_2(\tau)^{1/4}} \\
    &= \frac{1}{2^{1/6} (60 \cdot 2 \zeta(4))^{1/4}} \\
    &= \frac{3^{1/4}}{2^{2/3}\pi i},
\end{align*}
up to a 12th root of unity.

\begin{theorem}[$s = 0$ case]
\label{period-expr:s=0}
Suppose $\tau$ is in the connected component $C_{s,0}$ of the open set $|s|<1$. Then 
\begin{equation*}
  \omega_1 = \frac{2^{1/2}\pi i}{3 \tau} \vphantom{1em}_2 F_1 \left( \frac{1}{6}, \frac{5}{6}; 1; s \right)
\end{equation*}
up to a 12th root of unity, where $s = s(\tau)$.
\end{theorem}
\begin{proof}
Using the identity $\Delta(\tau) = \omega_1^{12} \Delta(E_s)$ in \eqref{delta-expr:s=0}, we obtain 
\begin{align*}
  \omega_1 & = \frac{2^{2/3} \pi i}{3^{1/4} 
  \tau}  \Delta(E_s)^{-\frac{1}{12}} s^{1/12} (1-s)^{1/12} \vphantom{1em}_2 F_1 \left( \frac{1}{6}, \frac{5}{6}; 1; s \right) \\
  & = \frac{2^{1/2}\pi i}{3 \tau} \vphantom{1em}_2 F_1 \left( \frac{1}{6}, \frac{5}{6}; 1; s \right).
\end{align*}
\end{proof}

\begin{theorem}[$s = 0$ case]
\label{period-expr-1:s=0}
Suppose $\tau$ is in the connected component $C_{s/(s-1),0}$ of the open set $|s|<|s-1|$. Then 
\begin{equation*}
  \omega_1 = \frac{2^{1/2}\pi i}{3 \tau} (1-s)^{-1/6} \vphantom{1em}_2 F_1 \left( \frac{1}{6}, \frac{1}{6}; 1; \frac{s}{s-1} \right)
\end{equation*}
up to a 12th root of unity, where $s = s(\tau)$.
\end{theorem}
\begin{proof}
The proof of this result is analogous to the above theorem; the only difference being the choice of Kummer solution (corresponding to $m=3, 19$ in \cite[Table I]{Dwork}) for obtaining a hypergeometric representation of $\widetilde{\omega}$.
\end{proof}

\subsection{Case 2B}

The fundamental domain we take in this case is
\begin{equation*}
  \mathcal{F}_{2B} = \{z \in \mathbb{C}: |z + 1/2| > 1/2, |z - 1/2| > 1/2, |{\rm Re}(z)| < 1/2\}.
\end{equation*}
Note that the transformation $\tau \mapsto \tau+1$ stabilizes the cusp at $\tau = \infty$, and $t$ is invariant under the action of $\Gamma_0(2)$. Hence
\begin{align*}
  (A + B\tau)\Delta(\tau)^{1/12} & = (A + B(\tau+1)) \Delta(\tau+1)^{1/12} \\
  & = (A + B(\tau+1)) e^{\pi i/6} \Delta(\tau)^{1/12},
\end{align*}
that is, around $\tau = \infty$,
\begin{equation*}
  (A + B\tau) = (A + B(\tau+1)) e^{\pi i/6}.
\end{equation*}
Letting $\tau \to \infty$, we obtain $B = 0$. Since around $t = \infty$, one of the solutions (corresponding to $m=9, 13$ in \cite[Table I]{Dwork}) is
\begin{equation*}
  \widetilde{\omega} = t^{-1/12} \vphantom{1em}_2 F_1 \left( \frac{1}{4}, \frac{1}{4}; 1; \frac{1}{t} \right),
\end{equation*}
we see that around $\tau = \infty$,
\begin{equation}
\label{delta-expr:t=infty}
  t^{-1/12} \vphantom{1em}_2 F_1 \left( \frac{1}{4}, \frac{1}{4} ; 1; \frac{1}{t} \right) = A \Delta(\tau)^{1/12}.
\end{equation}
Using the uniformizer expression for $t$ and the well-known $q$-expansion
$$\eta(\tau) = q^{1/24} (1 - q + O(q^2)), \quad q = e^{2\pi i \tau},$$
we find that
\begin{equation*}
  A = \frac{1}{\sqrt{2} \pi} e^{-\pi i/12}
\end{equation*}
up to a 12th root of unity.

\begin{theorem}[$t = \infty$ case]
\label{period-expr:t=infty}
Suppose $\tau$ is in the connected component $C_{1/t,\infty}$ of the open set $|t|>1$. Then 
\begin{equation*}
  \omega_1 = \frac{\sqrt{2} \pi}{3} t^{-1/4} (t-1)^{-1/4}  \vphantom{1em}_2 F_1 \left( \frac{1}{4}, \frac{1}{4} ; 1; \frac{1}{t} \right)
\end{equation*}
up to a 12th root of unity, where $t = t(\tau)$.
\end{theorem}
\begin{proof}
Using the identity $\Delta(\tau) = \omega_1^{12} \Delta(E_t)$ in \eqref{delta-expr:t=infty}, we obtain
\begin{align*}
  \omega_1 & = \sqrt{2} \pi e^{\pi i/12}  t^{-1/12}  \Delta(E_t)^{-1/12} \vphantom{1em}_2 F_1 \left( \frac{1}{4}, \frac{1}{4} ; 1; \frac{1}{t} \right) \\
& = \frac{\sqrt{2} \pi}{3}  t^{-1/4} (t-1)^{-1/4} \vphantom{1em}_2 F_1 \left( \frac{1}{4}, \frac{1}{4} ; 1; \frac{1}{t} \right).
\end{align*}
\end{proof}

\begin{theorem}[$t = \infty$ case]
\label{period-expr-1:t=infty}
Suppose $\tau$ is in the connected component $C_{1/(1-t),\infty}$ of the open set $|1-t|>1$. Then 
\begin{equation*}
  \omega_1 = \frac{\sqrt{2} \pi}{3} e^{\pi i/4} (t-1)^{-1/2}  \vphantom{1em}_2 F_1 \left( \frac{1}{4}, \frac{3}{4} ; 1; \frac{1}{1-t} \right)
\end{equation*}
up to a 12th root of unity, where $t = t(\tau)$.
\end{theorem}
\begin{proof}
The proof of this result is analogous to the above theorem; the only difference being the choice of Kummer solution (corresponding to $m=7, 8, 15, 16$ in \cite[Table I]{Dwork}) for obtaining a hypergeometric representation of $\widetilde{\omega}$.
\end{proof}

\subsection{Case 2C}

The fundamental domain we take in this case is
\begin{equation*}
  \mathcal{F}_{2C} = \{z \in \mathbb{C}: |z + 1| > 1, |z - 2| > 1, |z - 1/3| > 1/3, |z - 2/3| > 1/3, |{\rm Re}(z-1/2)| < 1\}.
\end{equation*}
Note that the transformation $\tau \mapsto \tau+2$ stabilizes the cusp at $\tau = \infty$, and $\lambda$ is invariant under the action of the congruence group $\Gamma(2)$. Hence
\begin{align*}
  (A + B\tau)\Delta(\tau)^{1/12} & = (A + B(\tau+2)) \Delta(\tau+2)^{1/12} \\
  & = (A + B(\tau+2)) e^{\pi i/3} \Delta(\tau)^{1/12},
\end{align*}
that is, around $\tau = \infty$,
\begin{equation*}
  (A + B\tau) = (A + B(\tau+2)) e^{\pi i/3}.
\end{equation*}
Letting $\tau \to \infty$, we obtain $B = 0$. Since around $\lambda = 0$, one of the solutions (corresponding to $m=1, 2, 17, 18$ in \cite[Table I]{Dwork}) is
\begin{equation*}
  \widetilde{\omega} = \lambda^{1/6} (1 - \lambda)^{1/6} {}_{2}F_{1}\bigg(\frac{1}{2}, \frac{1}{2}; 1; \lambda\bigg),
\end{equation*}
we see that around $\tau = \infty$,
\begin{equation}
\label{delta-expr:lambda=0}
  \lambda^{1/6} (1 - \lambda)^{1/6} {}_{2}F_{1}\bigg(\frac{1}{2}, \frac{1}{2}; 1; \lambda\bigg) = A \Delta(\tau)^{1/12}.
\end{equation}
Using the well-known $q$-expansion
$$\lambda(\tau) = 16q^{1/2} - 128q + 704q^{3/2} - O(q^2), \quad q = e^{2\pi i \tau},$$
we find that
\begin{equation*}
  A = \frac{2^{2/3}}{2\pi}
\end{equation*}
up to a 12th root of unity.

\begin{theorem}[$\lambda = 0$ case]
\label{period-expr:lambda=0}
Suppose $\tau$ is in the connected component $C_{\lambda,\infty}$ of the open set $|\lambda|<1$. Then
\begin{equation*}
  \omega_1 = \pi \vphantom{1em}_2 F_1 \left( \frac{1}{2}, \frac{1}{2} ; 1; \lambda \right)
\end{equation*}
up to a 12th root of unity, where $\lambda = \lambda(\tau)$.
\end{theorem}
\begin{proof}
Using the identity $\Delta(\tau) = \omega_1^{12} \Delta(E_\lambda)$ in \eqref{delta-expr:lambda=0}, we obtain
\begin{align*}
  \omega_1 & = 2^{1/3} \pi \lambda^{1/6} (1 - \lambda)^{1/6} \Delta(E_\lambda)^{-1/12} {}_{2}F_{1}\bigg(\frac{1}{2}, \frac{1}{2}; 1; \lambda\bigg) \\
& = \pi \vphantom{1em}_2 F_1 \left( \frac{1}{2}, \frac{1}{2} ; 1; \lambda \right).
\end{align*}
\end{proof}

\subsection{Case 3B}

The fundamental domain we take in this case is
\begin{equation*}
  \mathcal{F}_{3B} = \{z \in \mathbb{C}: |z + 1/3| > 1/3, |z - 1/3| > 1/3, |{\rm Re}(z)| < 1/2\}.
\end{equation*}
Note that the transformation $\tau \mapsto \tau+1$ stabilizes the cusp at $\tau = \infty$, and $u$ is invariant under the action of the congruence group $\Gamma_0(3)$. Hence
\begin{align*}
  (A + B\tau)\Delta(\tau)^{1/12} & = (A + B(\tau+1)) \Delta(\tau+1)^{1/12} \\
  & = (A + B(\tau+1)) e^{\pi i/6} \Delta(\tau)^{1/12},
\end{align*}
that is, around $\tau = \infty$,
\begin{equation*}
  (A + B\tau) = (A + B(\tau+1)) e^{\pi i/6}.
\end{equation*}
Letting $\tau \to \infty$, we obtain $B = 0$. Since around $u = \infty$, one of the solutions (corresponding to $m=9, 13$ in \cite[Table I]{Dwork}) is
\begin{equation*}
  \widetilde{\omega} = u^{-1/12} \vphantom{1em}_2 F_1 \left( \frac{1}{3}, \frac{1}{3}; 1; \frac{1}{u} \right),
\end{equation*}
we see that around $\tau = \infty$,
\begin{equation}
\label{delta-expr:u=infty}
  u^{-1/12} \vphantom{1em}_2 F_1 \left( \frac{1}{3}, \frac{1}{3}; 1; \frac{1}{u} \right) = A \Delta(\tau)^{1/12}.
\end{equation}
Using the uniformizer expression for $u$ and the well-known $q$-expansion
$$\eta(\tau) = q^{1/24} (1 - q + O(q^2)), \quad q = e^{2\pi i \tau},$$
we find that
\begin{equation*}
  A = \frac{3^{1/4}}{2 \pi} e^{-\pi i/12}
\end{equation*}
up to a 12th root of unity.

\begin{theorem}[$u = \infty$ case]
\label{period-expr:u=infty}
Suppose $\tau$ is in the connected component $C_{1/u,\infty}$ of the open set $|u|>1$. Then
\begin{equation*}
  \omega_1 = \frac{\sqrt{2} \pi}{3}  u^{-1/3} (u-1)^{-1/6} \vphantom{1em}_2 F_1 \left( \frac{1}{3}, \frac{1}{3} ; 1; \frac{1}{u} \right)
\end{equation*}
up to a 12th root of unity, where $u = u(\tau)$.
\end{theorem}
\begin{proof}
Using the identity $\Delta(\tau) = \omega_1^{12} \Delta(E_u)$ in \eqref{delta-expr:u=infty}, we obtain
\begin{align*}
  \omega_1 & = \frac{2 \pi}{3^{1/4}} e^{\pi i/12}  u^{-1/12}  \Delta(E_u)^{-1/12} \vphantom{1em}_2 F_1 \left( \frac{1}{3}, \frac{1}{3} ; 1; \frac{1}{u} \right) \\
& = \frac{\sqrt{2} \pi}{3}  u^{-1/3} (u-1)^{-1/6} \vphantom{1em}_2 F_1 \left( \frac{1}{3}, \frac{1}{3} ; 1; \frac{1}{u} \right).
\end{align*}
\end{proof}

\begin{theorem}[$u = \infty$ case]
\label{period-expr-1:u=infty}
Suppose $\tau$ is in the connected component $C_{1/(1-u),\infty}$ of the open set $|1-u|>1$. Then
\begin{equation*}
  \omega_1 = \frac{\sqrt{2} \pi}{3} (u-1)^{-1/2} \vphantom{1em}_2 F_1 \left( \frac{1}{3}, \frac{2}{3} ; 1; \frac{1}{1-u} \right)
\end{equation*}
up to a 12th root of unity, where $u = u(\tau)$.
\end{theorem}
\begin{proof}
The proof of this result is analogous to the above theorem; the only difference being the choice of Kummer solution (corresponding to $m=7, 8, 15, 16$ in \cite[Table I]{Dwork}) for obtaining a hypergeometric representation of $\widetilde{\omega}$.
\end{proof}

\subsection{Case 2A}

If $\overline{\omega}(t)$ is a solution of the Picard-Fuchs differential equation for $\Gamma_0(2)$, then there are constants $A$ and $B$ such that
$$\overline{\omega}(t) = A\overline{\omega}_1(t) + B\overline{\omega}_2(t) = (A + B\tau) \sqrt{\frac{dt}{d\tau}}.$$

Note that the transformation $\tau \mapsto \tau+1$ stabilizes the cusp at $\tau = \infty$, and $t$ is invariant under the action of the congruence group $\Gamma_0(2)$. Hence
\begin{align*}
  (A + B\tau) \sqrt{t'(\tau)} & = (A + B(\tau+1)) \sqrt{t'(\tau+1)} \\
  & = - (A + B(\tau+1)) \sqrt{t'(\tau)},
\end{align*}
that is, around $\tau = \infty$,
\begin{equation*}
  (A + B\tau) = - (A + B(\tau+1)).
\end{equation*}
Letting $\tau \to \infty$, we obtain $B = 0$. Since around $t = \infty$, one of the solutions (corresponding to $m=9, 13$ in \cite[Table I]{Dwork}) is
\begin{equation*}
  \overline{\omega}(t) = t^{1/4} (1 - t)^{1/4} \vphantom{1em}_2 F_1 \left( \frac{1}{4}, \frac{1}{4}; 1; \frac{1}{t} \right),
\end{equation*}
we see that around $\tau = \infty$,
\begin{equation}
\label{eqn:hypergeometric-t'}
  t^{1/4} (1 - t)^{1/4} \vphantom{1em}_2 F_1 \left( \frac{1}{4}, \frac{1}{4}; 1; \frac{1}{t} \right) = A \sqrt{\frac{dt}{d\tau}}.
\end{equation}

If $\overline{\omega}(v)$ is a solution of the Picard-Fuchs differential equation for $\Gamma_0^+(2)$, then there are constants $C$ and $D$ such that
$$\overline{\omega}(v) = C\overline{\omega}_1(v) + D\overline{\omega}_2(v) = (C + D\tau) \sqrt{\frac{dv}{d\tau}}.$$

Similar to the above case, note that the transformation $\tau \mapsto \tau+1$ stabilizes the cusp at $\tau = \infty$, and $v$ is invariant under the action of the congruence group $\Gamma_0^+(2)$. Hence
\begin{align*}
  (C + D\tau) \sqrt{v'(\tau)} & = (C + D(\tau+1)) \sqrt{v'(\tau+1)} \\
  & = - (C + D(\tau+1)) \sqrt{v'(\tau)},
\end{align*}
that is, around $\tau = \infty$,
\begin{equation*}
  (C + D\tau) = - (C + D(\tau+1)).
\end{equation*}
Letting $\tau \to \infty$, we obtain $D = 0$. Since around $v = \infty$, one of the solutions (corresponding to $m=9, 13$ in \cite[Table I]{Dwork}) is
\begin{equation*}
  \overline{\omega}(v) = v^{1/4} (1 - v)^{1/4} \vphantom{1em}_2 F_1 \left( \frac{1}{8}, \frac{3}{8}; 1; \frac{1}{v} \right),
\end{equation*}
we see that around $\tau = \infty$,
\begin{equation}
\label{eqn:hypergeometric-v'}
  v^{1/4} (1 - v)^{1/4} \vphantom{1em}_2 F_1 \left( \frac{1}{8}, \frac{3}{8}; 1; \frac{1}{v} \right) = D \sqrt{\frac{dv}{d\tau}}.
\end{equation}

\begin{lemma}
\label{lemma:subset:v-t}
The domain $C_{1/v,\infty}$ is a subset of domain $C_{1/t,\infty}$ for the following choice of fundamental domains for $\Gamma_0(2)$ and $\Gamma_0^+(2)$, respectively.
\begin{enumerate}
\item $\mathcal{F}_{2B} = \{z \in \mathbb{C}: |z + 1/2| > 1/2, |z - 1/2| > 1/2$ and $|{\rm Re} (z)| < 1/2\}$
\item $\mathcal{F}_{2A} = \{z \in \mathbb{C}: |z|^2 > 1/2$ and $|{\rm Re} (z)| < 1/2\}$
\end{enumerate}
\end{lemma}
\begin{proof}
The value of the uniformizer $t$ transforms according to the identity
\begin{equation}
\label{eqn:t-fricke-trans}
    t\left(-\frac{1}{2\tau}\right) = \frac{1}{t(\tau)}.
\end{equation}
 under the Fricke involution $\tau \mapsto -\frac{1}{2\tau}$. From the definition of $\eta(\tau)$, we have that \begin{equation*}
    \eta(-\overline{\tau}) = \overline{\eta(\tau)}.
\end{equation*}
Combining this with the uniformizer expression for $t(\tau)$, we obtain
\begin{equation}
\label{eqn:t-ref-trans}
    t(-\overline{\tau}) = \overline{t(\tau)}.
\end{equation}
Let us define $\gamma$ as the closed arc $\{z \in \mathbb{C}: |z|^2 = 1/2$ and $|{\rm Re} (z)| \le 1/2\}$. Using \eqref{eqn:t-fricke-trans} and \eqref{eqn:t-ref-trans}, we see that for $\tau \in \gamma$, we have
\begin{equation*}
    |t(\tau)| = 1.
\end{equation*}
Let $\mathfrak{H}^* = \mathfrak{H} \cup \mathbb{P}^1(\mathbb{Q})$ and $\mathbb{D} = \{z \in \mathbb{C}: |z| < 1\}$. The uniformizer $t$ gives a complex analytic isomorphism of Riemann surfaces $t: \Gamma_0(2) \backslash \mathfrak{H}^* \to \mathbb{P}^1(\mathbb{C})$. Identifying $\gamma$ as a simple closed curve $\widetilde{\gamma}$ in $\Gamma_0(2) \backslash \mathfrak{H}^*$, we observe that the uniformizer $t$ restricted to $\widetilde{\gamma}$ defines a continuous map $t \mid_{\widetilde{\gamma}}$ from $\widetilde{\gamma}$ into the unit circle. The map must have a non-zero degree since it is injective. Therefore, the map is surjective. Thus, $t \mid_{\widetilde{\gamma}}$ is a homeomoprphism from $\widetilde{\gamma}$ to the unit circle. This implies that $t(\mathcal{F}_{2A}) = \mathbb{C} \backslash \overline{\mathbb{D}}$ or $t(\mathcal{F}_{2A}) = \mathbb{D}$. But since $t(\tau) \to \infty$ as $\tau \to i\infty$, $t(\mathcal{F}_{2A}) = \mathbb{C} \backslash \overline{\mathbb{D}}$ and therefore $C_{1/v,\infty} \subseteq \mathcal{F}_{2A} = C_{1/t,\infty}$.
\end{proof}

\begin{theorem}[$v = \infty$ case]
\label{v=infty}
Suppose $\tau$ is in the connected component $C_{1/v,\infty}$ of the open set $|v|>1$. Then
\begin{equation*}
  \left( \frac{t}{t - 1} \right)^{1/4} \vphantom{1em}_2 F_1 \left( \frac{1}{8}, \frac{3}{8}; 1; \frac{1}{v} \right) = \vphantom{1em}_2 F_1 \left( \frac{1}{4}, \frac{1}{4}; 1; \frac{1}{t} \right)
\end{equation*}
up to a 4th root of unity, where $v = v(\tau)$ and $t = t(\tau)$.
\end{theorem}
\begin{proof}
Dividing \eqref{eqn:hypergeometric-v'} by \eqref{eqn:hypergeometric-t'}, we obtain
\begin{equation*}
  v^{1/4} (1 - v)^{1/4} \vphantom{1em}_2 F_1 \left( \frac{1}{8}, \frac{3}{8}; 1; \frac{1}{w} \right) = H t^{1/4} (1 - t)^{1/4} \vphantom{1em}_2 F_1 \left( \frac{1}{4}, \frac{1}{4}; 1; \frac{1}{t} \right) \sqrt{\frac{dv}{dt}},
\end{equation*}
where $H = D/A$ is a constant. Using the modular relation for case 2A and simplifying, we obtain
\begin{equation*}
  \left( \frac{t}{t - 1} \right)^{1/4} \vphantom{1em}_2 F_1 \left( \frac{1}{8}, \frac{3}{8}; 1; \frac{1}{v} \right) = H \vphantom{1em}_2 F_1 \left( \frac{1}{4}, \frac{1}{4}; 1; \frac{1}{t} \right).
\end{equation*}
Letting $\tau \to \infty$ in the above equation, we get
\begin{align*}
  H &= \lim_{\tau \to \infty} \left( \frac{t}{t - 1} \right)^{1/4}\\
  &= 1.
\end{align*}
\end{proof}

\subsection{Case 3A}

If $\overline{\omega}(u)$ is a solution of the Picard-Fuchs differential equation for $\Gamma_0(3)$, then there are constants $A$ and $B$ such that
$$\overline{\omega}(u) = A\overline{\omega}_1(u) + B\overline{\omega}_2(u) = (A + B\tau) \sqrt{\frac{du}{d\tau}}.$$

Note that the transformation $\tau \mapsto \tau+1$ stabilizes the cusp at $\tau = \infty$, and $u$ is invariant under the action of the congruence group $\Gamma_0(3)$. Hence
\begin{align*}
  (A + B\tau) \sqrt{u'(\tau)} & = (A + B(\tau+1)) \sqrt{u'(\tau+1)} \\
  & = - (A + B(\tau+1)) \sqrt{u'(\tau)},
\end{align*}
that is, around $\tau = \infty$,
\begin{equation*}
  (A + B\tau) = - (A + B(\tau+1)).
\end{equation*}
Letting $\tau \to \infty$, we obtain $B = 0$. Since around $u = \infty$, one of the solutions (corresponding to $m=9, 13$ in \cite[Table I]{Dwork}) is
\begin{equation*}
  \overline{\omega}(u) = u^{1/6} (1 - u)^{1/3} \vphantom{1em}_2 F_1 \left( \frac{1}{3}, \frac{1}{3}; 1; \frac{1}{u} \right),
\end{equation*}
we see that around $\tau = \infty$,
\begin{equation}
\label{eqn:hypergeometric-u'}
  u^{1/6} (1 - u)^{1/3} \vphantom{1em}_2 F_1 \left( \frac{1}{3}, \frac{1}{3}; 1; \frac{1}{u} \right) = A \sqrt{\frac{du}{d\tau}}.
\end{equation}

If $\overline{\omega}(w)$ is a solution of the Picard-Fuchs differential equation for $\Gamma_0^+(3)$, then there are constants $C$ and $D$ such that
$$\overline{\omega}(w) = C\overline{\omega}_1(w) + D\overline{\omega}_2(w) = (C + D\tau) \sqrt{\frac{dw}{d\tau}}.$$

Similar to the above case, note that the transformation $\tau \mapsto \tau+1$ stabilizes the cusp at $\tau = \infty$, and $w$ is invariant under the action of the congruence group $\Gamma_0^+(3)$. Hence
\begin{align*}
  (C + D\tau) \sqrt{w'(\tau)} & = (C + D(\tau+1)) \sqrt{w'(\tau+1)} \\
  & = - (C + D(\tau+1)) \sqrt{w'(\tau)},
\end{align*}
that is, around $\tau = \infty$,
\begin{equation*}
  (C + D\tau) = - (C + D(\tau+1)).
\end{equation*}
Letting $\tau \to \infty$, we obtain $D = 0$. Since around $w = \infty$, one of the solutions (corresponding to $m=9, 13$ in \cite[Table I]{Dwork}) is
\begin{equation*}
  \overline{\omega}(w) = w^{1/4} (1 - w)^{1/4} \vphantom{1em}_2 F_1 \left( \frac{1}{6}, \frac{1}{3}; 1; \frac{1}{w} \right),
\end{equation*}
we see that around $\tau = \infty$,
\begin{equation}
\label{eqn:hypergeometric-w'}
  w^{1/4} (1 - w)^{1/4} \vphantom{1em}_2 F_1 \left( \frac{1}{6}, \frac{1}{3}; 1; \frac{1}{w} \right) = D \sqrt{\frac{dw}{d\tau}}.
\end{equation}

\begin{lemma}
\label{lemma:subset:w-u}
The domain $C_{1/w,\infty}$ is a subset of domain $C_{1/u,\infty}$ for the following choice of fundamental domains for $\Gamma_0(3)$ and $\Gamma_0^+(3)$, respectively.
\begin{enumerate}
\item $\mathcal{F}_{3B} = \{z \in \mathbb{C}: |z + 1/3| > 1/3, |z - 1/3| > 1/3, |{\rm Re}(z)| < 1/2\}$
\item $\mathcal{F}_{3A} = \{z \in \mathbb{C}: |z|^2 > 1/3, |{\rm Re}(z)| < 1/2\}$
\end{enumerate}
\end{lemma}
\begin{proof}
The value of the uniformizer $u$ transforms according to the identity
\begin{equation}
\label{eqn:u-fricke-trans}
    u\left(-\frac{1}{3\tau}\right) = \frac{1}{u(\tau)}.
\end{equation}
under the Fricke involution $\tau \mapsto -\frac{1}{3\tau}$. From the definition of $\eta(\tau)$, we have that \begin{equation*}
    \eta(-\overline{\tau}) = \overline{\eta(\tau)}.
\end{equation*}
Combining this with the uniformizer expression for $u(\tau)$, we obtain
\begin{equation}
\label{eqn:u-ref-trans}
    u(-\overline{\tau}) = \overline{u(\tau)}.
\end{equation}
Let us define $\gamma$ as the closed arc $\{z \in \mathbb{C}: |z|^2 = 1/3$ and $|{\rm Re} (z)| \le 1/2\}$. Using \eqref{eqn:u-fricke-trans} and \eqref{eqn:u-ref-trans}, we see that for $\tau \in \gamma$, we have
\begin{equation*}
    |u(\tau)| = 1.
\end{equation*}
Let $\mathfrak{H}^* = \mathfrak{H} \cup \mathbb{P}^1(\mathbb{Q})$ and $\mathbb{D} = \{z \in \mathbb{C}: |z| < 1\}$. The uniformizer $u$ gives a complex analytic isomorphism of Riemann surfaces $u: \Gamma_0(3) \backslash \mathfrak{H}^* \to \mathbb{P}^1(\mathbb{C})$. Identifying $\gamma$ as a simple closed curve $\widetilde{\gamma}$ in $\Gamma_0(3) \backslash \mathfrak{H}^*$, we observe that the uniformizer $u$ restricted to $\widetilde{\gamma}$ defines a continuous map $u \mid_{\widetilde{\gamma}}$ from $\widetilde{\gamma}$ into the unit circle. The map must have a non-zero degree since it is injective. Therefore, the map is surjective. Thus, $u \mid_{\widetilde{\gamma}}$ is a homeomoprphism from $\widetilde{\gamma}$ to the unit circle. This implies that $u(\mathcal{F}_{3A}) = \mathbb{C} \backslash \overline{\mathbb{D}}$ or $u(\mathcal{F}_{3A}) = \mathbb{D}$. But since $u(\tau) \to \infty$ as $\tau \to i\infty$, $u(\mathcal{F}_{3A}) = \mathbb{C} \backslash \overline{\mathbb{D}}$ and therefore $C_{1/w,\infty} \subseteq \mathcal{F}_{3A} = C_{1/u,\infty}$.
\end{proof}

\begin{theorem}[$w = \infty$ case]
\label{w=infty}
Suppose $\tau$ is in the connected component $C_{1/w,\infty}$ of the open set $|w|>1$. Then
\begin{equation*}
  \left( \frac{u}{u - 1} \right)^{1/3} \vphantom{1em}_2 F_1 \left( \frac{1}{6}, \frac{1}{3}; 1; \frac{1}{w} \right) = \vphantom{1em}_2 F_1 \left( \frac{1}{3}, \frac{1}{3}; 1; \frac{1}{u} \right)
\end{equation*}
up to a 3rd root of unity, where $w = w(\tau)$ and $u = u(\tau)$.
\end{theorem}
\begin{proof}
Dividing \eqref{eqn:hypergeometric-w'} by \eqref{eqn:hypergeometric-u'}, we obtain
\begin{equation*}
  w^{1/4} (1 - w)^{1/4} \vphantom{1em}_2 F_1 \left( \frac{1}{6}, \frac{1}{3}; 1; \frac{1}{w} \right) = H u^{1/6} (1 - u)^{1/3} \vphantom{1em}_2 F_1 \left( \frac{1}{3}, \frac{1}{3}; 1; \frac{1}{u} \right) \sqrt{\frac{dw}{du}},
\end{equation*}
where $H = D/A$ is a constant. Using the modular relation for case 3A and simplifying, we obtain
\begin{equation*}
  \left( \frac{u}{u - 1} \right)^{1/3} \vphantom{1em}_2 F_1 \left( \frac{1}{6}, \frac{1}{3}; 1; \frac{1}{w} \right) = H \vphantom{1em}_2 F_1 \left( \frac{1}{3}, \frac{1}{3}; 1; \frac{1}{u} \right).
\end{equation*}
Letting $\tau \to \infty$ in the above equation, we get
\begin{align*}
  H &= \lim_{\tau \to \infty} \left( \frac{u}{u - 1} \right)^{1/3}\\
  &= 1.
\end{align*}
\end{proof}

\section{Chudnovsky-Ramanujan type precursor formulae}

For $\tau \in \mathfrak{H}$, let $q = e^{2\pi i \tau}$, and define
$$E_2^*(\tau) = E_2(\tau) - \frac{3}{\pi {\rm Im}(\tau)}, \quad s_2(\tau) = \frac{E_4(\tau)}{E_6(\tau)} E_2^*(\tau),$$
where
\begin{align*}
E_2(\tau) &= 1 - 24 \sum_{n = 1}^\infty n \frac{q^n}{1 - q^n}, \\
E_4(\tau) &= 1 + 240 \sum_{n = 1}^\infty n^3 \frac{q^n}{1 - q^n}, \\
E_6(\tau) &= 1 - 504 \sum_{n = 1}^\infty n^5 \frac{q^n}{1 - q^n}.
\end{align*}

It is known that $E_2^*(\tau)$ is an almost holomorphic modular form of weight $2$ \cite[Section 5.3]{Zagier}, and $E_4(\tau)$ and $E_6(\tau)$ are modular forms of weight $4$ and $6$, respectively. Therefore, $s_2(\tau)$ satisfies
\begin{equation*}
    s_2\left(\frac{a\tau + b}{c\tau+d}\right) = s_2(\tau)
\end{equation*}
for $\tau \in \{z \in \mathfrak{H} :  E_6(z) \neq 0 \}$ and $\begin{pmatrix} a & b \\ c & d \end{pmatrix} \in SL_2(\mathbb{Z})$.

From
\begin{equation*}
    g_2(\tau) = \frac{4 \pi^4 E_4(\tau)}{3} \quad \textrm{and} \quad g_3(\tau) = \frac{8 \pi^6 E_6(\tau)}{27},
\end{equation*}
and the fact that $g_2(\tau) = \omega_1^4 g_2$ and $g_3(\tau) = \omega_1^6 g_3$, we obtain
\begin{equation}
\label{eqn:3g_3/2g_2}
    \frac{3g_3}{2g_2} = \omega_1^2 \frac{\pi^2}{3} \frac{E_6(\tau)}{E_4(\tau)}.
\end{equation}

Consider the elliptic curve $E$ over $\mathbb{C}$ given by
$$E : y^2 = 4x^3 - g_2 x - g_3, \quad g_2^3 - 27g_3^2 \ne 0,$$
whose invariants $g_2 = A$ and $g_3 = B$ are functions of a parameter $\xi$ and whose periods and quasi-periods are $\omega = \omega_k$, $\eta = \eta_k$, with $k = 1, 2$, respectively. We established earlier in \eqref{eq:Bruns} that
\begin{empheq}[left=\empheqlbrace]{align}
\nonumber
\begin{split}
\frac{d\omega}{d\xi} &= - Q\omega - P\eta , \\
\frac{d\eta}{d\xi} &= R\omega + Q\eta.
\end{split}
\end{empheq}
Clearly,
$$\eta = -\frac{\omega' + Q\omega}{P}, \quad \omega' = \frac{d\omega}{d\xi}.$$
Combining this with \cite[Theorem 3.6]{Chen}, we obtain
$$\omega_1 \eta_1 {\rm Im}(\tau) - \omega_1^2 \left[{\rm Im}(\tau) \frac{3g_3}{2g_2} s_2(\tau)\right] = \pi,$$ or 
\begin{equation}
\label{eq:period}
-\omega_1 \frac{\omega_1' + Q\omega_1}{P} - \omega_1^2 \left[\frac{3g_3}{2g_2} s_2(\tau)\right] = \frac{\pi}{{\rm Im}(\tau)}.
\end{equation}
If $\omega_1 = k(\xi) F(\xi)$, where $F(\xi)$ is one of Kummer's solutions and $k(\xi)$ is some function of $\xi$, then $\omega_1' = k F' + k' F$. Therefore, \eqref{eq:period} can be recast as
\begin{equation}
\label{Chudnovsky-Ramanujan}
-k^2 F^2 \bigg\{\frac{k Q + k'}{k P} + \frac{3g_3}{2g_2} s_2(\tau)\bigg\} - \frac{k^2}{2P} \frac{d F^2}{d\xi}  = \frac{\pi}{{\rm Im} (\tau)}
\end{equation}
since $2FF' = (F^2)'$. This is equivalent to
\begin{equation}
\label{Chudnovsky-Ramanujan-2}
-k^2 F^2 \bigg\{\frac{k Q + k'}{k P} + \frac{\pi^2}{3} \bigg(E_2(\tau) - \frac{3}{\pi {\rm Im}(\tau)}\bigg)\bigg\} - \frac{k^2}{2P} \frac{d F^2}{d\xi}  = \frac{\pi}{{\rm Im} (\tau)}
\end{equation}
in view of equation \eqref{eqn:3g_3/2g_2}.

Suppose $\tau \in \mathfrak{H}$ satisfies $a\tau^2 + b\tau + c = 0$ for mutually coprime integers $a$, $b$, $c$ such that $a > 0$ and $-d = b^2 - 4ac$, that is,
\begin{equation}
\label{Chudnovsky}
\tau = \frac{-b + \sqrt{-d}}{2a}.
\end{equation}
Either \eqref{Chudnovsky-Ramanujan} or \eqref{Chudnovsky-Ramanujan-2} now yields a formula for $1/\pi$.

\subsection{Case 1B}

\begin{lemma}
\label{lemma:s'/s}
We have
\begin{equation*}
  \frac{s'}{s} = - \frac{9\omega_1^2}{\pi i} (s-1),
\end{equation*}
where $s = s(\tau)$ and $s' = \frac{ds}{d\tau}$.
\end{lemma}
\begin{proof}
Using the chain rule of differentiation for the modular relation
\begin{equation*}
    J = \frac{1}{4s(1-s)},
\end{equation*}
and dividing the equation by $J$ on both sides, we obtain
\begin{equation*}
    \frac{J'}{J} = - \left(\frac{2s-1}{s-1}\right) \frac{s'}{s}.
\end{equation*}
The quotient on the RHS of the above equation is given by \cite[Section 2.4]{Chen}
\begin{equation*}
  \frac{J'}{J} = -2\pi i\frac{E_6(\tau)}{E_4(\tau)}.
\end{equation*}
Hence, it follows that
\begin{equation*}
    \frac{s'}{s} = 2 \pi i \left(\frac{s-1}{2s-1}\right) \frac{E_6(\tau)}{E_4(\tau)}.
\end{equation*}
From
\begin{equation*}
    g_2(\tau) = \frac{4 \pi^4 E_4(\tau)}{3} \quad \textrm{and} \quad g_3(\tau) = \frac{8 \pi^6 E_6(\tau)}{27},
\end{equation*}
and the fact that $g_2(\tau) = \omega_1^4 g_2$ and $g_3(\tau) = \omega_1^6 g_3$, we obtain
\begin{equation*}
    \frac{s'}{s} = - \frac{9\omega_1^2}{\pi i} \left(\frac{s-1}{2s-1}\right) \frac{g_3}{g_2}.
\end{equation*}
Lastly, it follows using the expressions for $g_2$ and $g_3$ that
\begin{equation*}
  \frac{s'}{s} = - \frac{9\omega_1^2}{\pi i} (s-1).
\end{equation*}
\end{proof}

\begin{theorem}[$s = 0$ case]
\label{thm:s=0}
If $\tau$ is as in \eqref{Chudnovsky} and lies in $C_{s,0}$, then
\begin{equation*}
  \frac{2s-1}{6} F^2 \left[-1 + s_2(\tau)\right] + s(1-s) \frac{d}{ds}F^2 = \frac{a \tau^2}{\pi \sqrt{d}} - \frac{\tau i}{\pi},
\end{equation*}
where $F = \vphantom{1em}_2 F_1 \left( \frac{1}{6}, \frac{5}{6} ; 1; s \right)$ and $s = s(\tau)$.
\end{theorem}
\begin{proof}
Recall from theorem \ref{period-expr:s=0} that
\begin{equation*}
  \omega_1 = \frac{2^{1/2}\pi i}{3 \tau} \vphantom{1em}_2 F_1 \left( \frac{1}{6}, \frac{5}{6}; 1; s \right).
\end{equation*}
Using this expression in \eqref{eq:period}, we obtain
\begin{equation*}
  F^2 \left[-\frac{3g_3}{2g_2} + \frac{3g_3}{2g_2} s_2(\tau) - \frac{18s(1-s)}{\tau \frac{ds}{d\tau}}\right] + 9s(1-s) \frac{d}{ds}F^2 = \frac{9a \tau^2}{\pi \sqrt{d}},
\end{equation*}
Now using Lemma \ref{lemma:s'/s}, we obtain the desired result.
\end{proof}

\begin{theorem}[$s = 0$ case]
\label{thm-1:s=0}
If $\tau$ is as in \eqref{Chudnovsky} and lies in $C_{s/(s-1),0}$, then
\begin{equation*}
  \frac{F^2}{6} \left[1 + (2s-1) s_2(\tau)\right] + s(1-s) \frac{d}{ds}F^2 = \left[\frac{a \tau^2}{\pi \sqrt{d}} - \frac{\tau i}{\pi}\right] (1-s)^{1/3},
\end{equation*}
where $F = \vphantom{1em}_2 F_1 \left( \frac{1}{6}, \frac{1}{6} ; 1; \frac{s}{s-1} \right)$ and $s = s(\tau)$.
\end{theorem}
\begin{proof}
Recall from theorem \ref{period-expr-1:s=0} that
\begin{equation*}
  \omega_1 = \frac{2^{1/2}\pi i}{3 \tau} (1-s)^{-1/6} \vphantom{1em}_2 F_1 \left( \frac{1}{6}, \frac{1}{6}; 1; \frac{s}{s-1} \right).
\end{equation*}
Using this expression in \eqref{eq:period}, we obtain
\begin{equation*}
  F^2 \left[-\frac{3g_3}{2g_2} + \frac{3g_3}{2g_2} s_2(\tau) - \frac{18s(1-s)}{\tau \frac{ds}{d\tau}} + 3s\right] + 9s(1-s) \frac{d}{ds}F^2 = \frac{9a \tau^2}{\pi \sqrt{d}} (1-s)^{1/3},
\end{equation*}
Now using Lemma \ref{lemma:s'/s}, we obtain the desired result.
\end{proof}

\subsection{Case 2B}

\begin{theorem}[$t = \infty$ case]
\label{thm:t=infty}
If $\tau$ is as in \eqref{Chudnovsky} and lies in $C_{1/t,\infty}$, then
\begin{equation*}
  \frac{F^2}{6} \left[1 - \left(\frac{t+8}{t-4}\right) s_2(\tau)\right] - t \frac{d}{dt} F^2 = \frac{a}{\pi \sqrt{d}} \left(\frac{t}{t-1}\right)^{1/2},
\end{equation*}
where $F = \vphantom{1em}_2 F_1 \left( \frac{1}{4}, \frac{1}{4} ; 1; \frac{1}{t} \right)$ and $t = t(\tau)$.
\end{theorem}
\begin{proof}
Recall from Theorem~\ref{period-expr:t=infty} that
\begin{equation*}
  \omega_1 = \frac{2 \pi}{3 \cdot 2^{1/2}} t^{-1/4} (t-1)^{-1/4}  \vphantom{1em}_2 F_1 \left( \frac{1}{4}, \frac{1}{4} ; 1; \frac{1}{t} \right).
\end{equation*}
Substituting the above expression for $\omega_1$ into \eqref{eq:period} gives the desired result.
\end{proof}

\begin{theorem}[$t = \infty$ case]
\label{thm-1:t=infty}
If $\tau$ is as in \eqref{Chudnovsky} and lies in $C_{1/(1-t),\infty}$, then
\begin{equation*}
  \frac{F^2}{6} \left[(t+2) - \frac{(t-1)(t+8)}{t-4} s_2(\tau)\right] - t(t-1) \frac{d}{dt} F^2 = \frac{a}{\pi \sqrt{d}} (t-1),
\end{equation*}
where $F = \vphantom{1em}_2 F_1 \left( \frac{1}{4}, \frac{3}{4} ; 1; \frac{1}{1-t} \right)$ and $t = t(\tau)$.
\end{theorem}
\begin{proof}
Recall from Theorem~\ref{period-expr-1:t=infty} that
\begin{equation*}
  \omega_1 = \frac{2 \pi}{3 \cdot 2^{1/2}} e^{\pi i/4} (t-1)^{-1/2}  \vphantom{1em}_2 F_1 \left( \frac{1}{4}, \frac{3}{4} ; 1; \frac{1}{1-t} \right).
\end{equation*}
Substituting the above expression for $\omega_1$ into \eqref{eq:period} gives the desired result.
\end{proof}

\subsection{Case 2C}

\begin{theorem}[$\lambda = 0$ case]
\label{thm:lambda=0}
If $\tau$ is as in \eqref{Chudnovsky} and lies in $C_{\lambda,\infty}$, then 
\begin{equation*}
  (1 - 2\lambda) \frac{F^2}{6} \left[2 + \frac{(\lambda + 1)(\lambda - 2)}{\lambda^2 - \lambda + 1} s_2(\tau)\right] + \lambda (1 - \lambda) \frac{dF^2}{d\lambda} = \frac{2a}{\pi \sqrt{d}},
\end{equation*}
where $F = \vphantom{1em}_2 F_1 \left( \frac{1}{2}, \frac{1}{2} ; 1; \lambda \right)$ and $\lambda = \lambda(\tau)$.
\end{theorem}
\begin{proof}
Recall from theorem~\ref{period-expr:lambda=0} that
\begin{equation*}
  \omega_1 = \pi \vphantom{1em}_2 F_1 \left( \frac{1}{2}, \frac{1}{2} ; 1; \lambda \right).
\end{equation*}
Substituting the above expression for $\omega_1$ into \eqref{eq:period} gives the desired result.
\end{proof}

\subsection{Case 3B}

\begin{theorem}[$u = \infty$ case]
\label{thm:u=infty}
If $\tau$ is as in \eqref{Chudnovsky} and lies in $C_{1/u,\infty}$, then
\begin{equation*}
  \frac{F^2}{6} \left[(u-1) - \frac{u^2+18u-27}{u-9} s_2(\tau)\right] - u(u-1) \frac{d}{du}F^2 = \frac{a}{\pi \sqrt{d}}u^{2/3}(u-1)^{1/3},
\end{equation*}
where $F = \vphantom{1em}_2 F_1 \left( \frac{1}{3}, \frac{1}{3} ; 1; \frac{1}{u} \right)$ and $u = u(\tau)$.
\end{theorem}
\begin{proof}
Recall from Theorem~\ref{period-expr:u=infty} that
\begin{equation*}
  \omega_1 = \frac{\sqrt{2} \pi}{3}  u^{-1/3} (u-1)^{-1/6} \vphantom{1em}_2 F_1 \left( \frac{1}{3}, \frac{1}{3} ; 1; \frac{1}{u} \right).
\end{equation*}
Substituting the above expression for $\omega_1$ into \eqref{eq:period} gives the desired result.
\end{proof}

\begin{theorem}[$u = \infty$ case]
\label{thm-1:u=infty}
If $\tau$ is as in \eqref{Chudnovsky} and lies in $C_{1/(1-u),\infty}$, then
\begin{equation*}
  \frac{F^2}{6} \left[(u+3) - \frac{u^2+18u-27}{u-9} s_2(\tau)\right] - u(u-1) \frac{d}{du}F^2 = \frac{a}{\pi \sqrt{d}}(u-1),
\end{equation*}
where $F = \vphantom{1em}_2 F_1 \left( \frac{1}{3}, \frac{2}{3} ; 1; \frac{1}{1-u} \right)$ and $u = u(\tau)$.
\end{theorem}
\begin{proof}
Recall from Theorem~\ref{period-expr-1:u=infty} that
\begin{equation*}
  \omega_1 = \frac{\sqrt{2} \pi}{3} (u-1)^{-1/2} \vphantom{1em}_2 F_1 \left( \frac{1}{3}, \frac{2}{3} ; 1; \frac{1}{1-u} \right).
\end{equation*}
Substituting the above expression for $\omega_1$ into \eqref{eq:period} gives the desired result.
\end{proof}

\subsection{Case 2A}

\begin{theorem}[$v = \infty$ case]
\label{thm:v=infty}
If $\tau$ is as in \eqref{Chudnovsky} and lies in $C_{1/v,\infty}$, then
\begin{equation*}
  \frac{F^2}{6} \left[\frac{t + 2}{t + 1} - \frac{ (t-1)(t+8)}{(t-4)(t+1)} s_2(\tau)\right] - v \frac{d}{dv}F^2 = \frac{a}{\pi \sqrt{d}} \frac{t-1}{t+1},
\end{equation*}
where $F = \vphantom{1em}_2 F_1 \left( \frac{1}{8}, \frac{3}{8} ; 1; \frac{1}{v} \right)$, $v = v(\tau)$, and $t = t(\tau)$.
\end{theorem}
\begin{proof}
Let $F = \vphantom{1em}_2 F_1 \left( \frac{1}{8}, \frac{3}{8} ; 1; \frac{1}{v} \right)$ and $G = \vphantom{1em}_2 F_1 \left( \frac{1}{4}, \frac{1}{4}; 1; \frac{1}{t} \right)$. Recall from Theorem~\ref{v=infty} that
\begin{equation*}
  \left( \frac{t}{t - 1} \right)^{1/4} F = G.
\end{equation*}
Squaring both sides gives
\begin{equation*}
  \left( \frac{t}{t - 1} \right)^{1/2} F^2 = G^2.
\end{equation*}
Differentiating both sides of the above equation with respect to $t$, we obtain
\begin{equation}
\label{eqn:2A-proof-1}
  -\frac{1}{2} \left( \frac{1}{t(t - 1)^3} \right)^{1/2} F^2 + \left( \frac{t}{t - 1} \right)^{1/2} \frac{d}{dt} F^2  = \frac{d}{dt} G^2.
\end{equation}
We have
\begin{equation*}
  v \frac{d}{dt} F^2 = \frac{dv}{dt} \left(v \frac{d}{dv} F^2\right) ,
\end{equation*}
which is equivalent to
\begin{equation*}
  \frac{d}{dt} F^2 = \frac{t + 1}{t (t - 1)} \left(v \frac{d}{dv} F^2\right).
\end{equation*}
Using this in \eqref{eqn:2A-proof-1}, we obtain
\begin{equation}
\label{eqn:2A-proof-2}
  -\frac{1}{2(t + 1)} F^2 + v \frac{d}{dv} F^2  = \frac{(t(t - 1)^3)^{1/2}}{t+1} \frac{d}{dt} G^2.
\end{equation}
From Theorem~\ref{thm:t=infty}, we know that
\begin{equation*}
  \frac{G^2}{6} \left[1 - \left(\frac{t+8}{t-4}\right) s_2(\tau)\right] - t \frac{d}{dt} G^2 = \frac{a}{\pi \sqrt{d}} \left(\frac{t}{t-1}\right)^{1/2}.
\end{equation*}
Multiplying the above equation by $\frac{1}{t + 1} \left(\frac{(t - 1)^3}{t}\right)^{1/2}$ and using \eqref{eqn:2A-proof-2}, we obtain
\begin{equation*}
  \frac{t - 1}{t + 1} \frac{F^2}{6} \left[1 - \left(\frac{t+8}{t-4}\right) s_2(\tau)\right] + \frac{1}{2(t + 1)} F^2 - v \frac{d}{dv} F^2 = \frac{a}{\pi \sqrt{d}} \frac{t - 1}{t + 1},
\end{equation*}
which can be recast to obtain the desired result.
\end{proof}

\subsection{Case 3A}

\begin{theorem}[$w = \infty$ case]
\label{thm:w=infty}
If $\tau$ is as in \eqref{Chudnovsky} and lies in $C_{1/w,\infty}$, then
\begin{equation*}
  \frac{F^2}{6} \left[\frac{u + 3}{u + 1} - \frac{u^2 + 18u - 27}{(u - 9) (u + 1)} s_2(\tau)\right] - w \frac{d}{dw} F^2 = \frac{a}{\pi \sqrt{d}} \frac{u - 1}{u + 1},
\end{equation*}
where $F = \vphantom{1em}_2 F_1 \left( \frac{1}{6}, \frac{1}{3} ; 1; \frac{1}{w} \right)$, $w = w(\tau)$, and $u = u(\tau)$.
\end{theorem}
\begin{proof}
Let $F = \vphantom{1em}_2 F_1 \left( \frac{1}{6}, \frac{1}{3} ; 1; \frac{1}{w} \right)$ and $G = \vphantom{1em}_2 F_1 \left( \frac{1}{3}, \frac{1}{3} ; 1; \frac{1}{u} \right)$. Recall from Theorem~\ref{w=infty} that
\begin{equation*}
  \left( \frac{u}{u - 1} \right)^{1/3} F = G.
\end{equation*}
Squaring both sides gives
\begin{equation*}
  \left( \frac{u}{u - 1} \right)^{2/3} F^2 = G^2.
\end{equation*}
Differentiating the above equation on both sides with respect to $u$, we obtain
\begin{equation}
\label{eqn:3A-proof-1}
  -\frac{2}{3} \left( \frac{1}{u(u - 1)^5} \right)^{1/3} F^2 + \left( \frac{u}{u - 1} \right)^{2/3} \frac{d}{du} F^2  = \frac{d}{du} G^2.
\end{equation}
We have
\begin{equation*}
  w \frac{d}{du} F^2 = \frac{dw}{du} \left(w \frac{d}{dw} F^2\right), 
\end{equation*}
which is equivalent to
\begin{equation*}
  \frac{d}{du} F^2 = \frac{u + 1}{u (u - 1)} \left(w \frac{d}{dw} F^2\right). 
\end{equation*}
Using this in \eqref{eqn:3A-proof-1}, we obtain
\begin{equation}
\label{eqn:3A-proof-2}
  -\frac{2}{3(u + 1)} F^2 + w \frac{d}{dw} F^2  = \frac{(u (u - 1)^5)^{1/3}}{u + 1} \frac{d}{du} G^2.
\end{equation}
From Theorem~\ref{thm:u=infty}, we know that
\begin{equation*}
    \frac{G^2}{6} \left[(u-1) - \frac{u^2+18u-27}{u-9} s_2(\tau)\right] - u(u-1) \frac{d}{du} G^2 = \frac{a}{\pi \sqrt{d}}u^{2/3}(u-1)^{1/3}.
\end{equation*}
Multiplying the above equation equation by $\frac{1}{u + 1} \left(\frac{u - 1}{u}\right)^{2/3}$ and using \eqref{eqn:3A-proof-2}, we obtain
\begin{equation*}
  \frac{1}{u + 1} \frac{F^2}{6} \left[(u-1) - \frac{u^2+18u-27}{u-9} s_2(\tau)\right] + \frac{2}{3(u + 1)} F^2 - w \frac{d}{dw} F^2 = \frac{a}{\pi \sqrt{d}} \frac{u - 1}{u + 1},
\end{equation*}
which can be recast to obtain the desired result.
\end{proof}

\section{Singular values}

It is known that $s_2(\tau)$ is rational at $\tau = \sqrt{-N}$ for $N = 2, 3, 4, 7$ and at $\tau = \frac{-1 + \sqrt{-N}}{2}$ for $N = 7, 11, 19, 27, 43, 67,\allowbreak 163$. This follows from \cite[Lemma 4.1]{Chudnovsky2} which is based on \cite[\S 3, Chapter 6]{Weil}. Unfortunately, we are not aware of any complete published proof of \cite[Lemma 4.1]{Chudnovsky2}, except in \cite[Theorem 4.2]{Berndt} where a proof is given for certain cases using different methods. Another persistent obstruction evident in past literature has been the lack of a direct method to rigorously establish the values of $s_2(\tau)$ for an imaginary quadratic irrational $\tau$.

In this section, we will give a complete proof of \cite[Lemma 4.1]{Chudnovsky2} and use it to give a general algorithm to rigorously confirm the values $s_2(\tau)$ for imaginary quadratic irrationals $\tau$ in the complex upper half plane. The proof uses results from \cite[\S 3, Chapter 6]{Weil} as well the theory of complex multiplication of elliptic curves.

For a lattice $\Lambda$ in $\C$, we let
\begin{align*}
    G_k(z,\Lambda) & = \sum_{w \in \Lambda} (z+w)^{-k} \\
    e_k(\Lambda) & = \sum_{w \in \Lambda -\{0\}} w^{-k} \\
\end{align*}
using the notation in \cite{Weil}, except our $G_k(z,\Lambda)$ is his $E_k(z,\Lambda)$.


\begin{theorem}
\label{height}
Suppose $\tau \in \mathfrak{H}$ satisfies $a\tau^2 + b \tau + c$ for mutually coprime integers $a,b,c$ such that $a > 0$ and $-d = b^2 - 4 a c$. Then $s_2(\tau) \in L = \Q(j(\tau))$ and there is an explicit positive integer $M$ such that $M s_2(\tau)$ lies in the ring of integers $\mathcal{O}_L$ of $L$.
\end{theorem}
\begin{proof}
Let $\Lambda_\tau = \Z + \Z \tau$ and $\Lambda = \Z u + \Z v$ where $\tau = v/u$. In the notation of \cite{Weil}, we have the identity
\begin{equation*}
    \frac{\bar u}{A u} = \frac{\pi}{u^2 {\rm Im}(\tau)}.
\end{equation*}
Thus, we obtain that
\begin{align}
\label{s2-invariance}
    s_2(\tau) & = \frac{E_4(\tau)}{E_6(\tau)}\left( E_2(\tau) - \frac{3}{\pi {\rm Im}(\tau)} \right)
     = \frac{E_4(\Lambda_\tau)}{E_6(\Lambda_\tau)}\left( E_2(\Lambda_\tau) - \frac{3}{\pi {\rm Im}(\tau)} \right), \\
    \notag & = \frac{E_4(\Lambda)}{E_6(\Lambda)}\left( E_2(\Lambda) - \frac{3}{\pi u^2 {\rm Im}(\tau)} \right)  = \frac{2}{7} \frac{e_4(\Lambda)}{e_6(\Lambda)}\left( e_2(\Lambda) - \frac{\bar u}{A u} \right), \\
    \notag & = \frac{2}{7} \frac{e_4(\Lambda)}{e_6(\Lambda)} e_2^*(\Lambda),
\end{align}
again using the notation of \cite{Weil}. Let $K = \Q(\tau)$, and $\omega = a \tau$, which is an element of the order of discriminant $-d$. 

The element $\omega$ is a complex multiplier for $\Lambda_\tau$ and $\Lambda$. Using \cite[(7) Chap. IV. \S3]{Weil}, we can express
\begin{align}
\label{cm-identity}
    \omega (\omega - \bar \omega) e_2^*(\Lambda) = \sum_{r \in \Lambda/\omega \Lambda - \left\{ 0 \right\}} (G_2(r/\omega,\Lambda) - e_2(\Lambda)), \\
\label{cm-identity-bar}
    \bar \omega (\bar \omega - \omega) e_2^*(\Lambda) = \sum_{r \in \Lambda/\bar \omega \Lambda - \left\{ 0 \right\}} (G_2(r/\bar \omega,\Lambda) - e_2(\Lambda)).
\end{align}
Now recall that
\begin{equation}
    G_2(z,\Lambda) - e_2(\Lambda) = \wp_\Lambda(z)
\end{equation}
is the Weierstrass $\wp$-function $\wp_\Lambda(x)$ \cite[(16) Chap. III. \S8]{Weil}, which gives the complex uniformization of the elliptic curve
\begin{align}
    \notag \C/\Lambda & \rightarrow E: y^2 = 4 x^3 - g_2(\Lambda) x - g_3(\Lambda) \\
    z & \mapsto (\wp_\Lambda(z) : \wp_\Lambda'(z) : 1)
\end{align}
where $g_2(\Lambda) = 60 \cdot e_4$, $g_3(\Lambda) = 140 \cdot e_6$. We have that $\End(E) \cong \mathcal{O}$ for some order in $K$, and for an ideal $I \subseteq \mathcal{O}$, define
\begin{equation}
    E[I] = \left\{ P \in E: u I = 0 \text{ for all } u \in I \right\}
\end{equation}
to be the subgroup of $I$-torsion points of $E$.

It follows that the right hand side of \eqref{cm-identity} is the sum
\begin{equation}
\label{trace-sum}
    \sum_{P \in E[(\omega)] - \left\{ 0_E \right\}} x(P),
\end{equation}
where $x(P)$ denotes the $x$-coordinate of a point $P$ on $E$, $E[(\omega)]$ denotes the $(\omega)$-torsion points on $E$, $(\omega)$ is the ideal of $\mathcal{O}$ generated by $\omega$, and $0_E = (0:1:0)$. Similarly for $\omega$ replaced by $\bar \omega$.

Choose $E$ so it is defined over $L = \Q(j(E))$, that is, $g_2(\Lambda), g_3(\Lambda) \in \Q(j(E))$, where $j(E) = j(\Lambda) = j(\tau)$ is the $j$-invariant of $E$. As $E[(\omega)]$ is invariant under the action of the absolute Galois group of the compositum $LK$, we see that the sum \eqref{trace-sum} lies in $K(g_2, g_3) = K(e_4,e_6)$, from which it follows that $s_2(\tau) \in K(j(\tau))$ by \eqref{s2-invariance} and \eqref{cm-identity}.

To show that $s_2(\tau) \in \Q(j(\tau))$ in fact, we note from \cite[Theorem 2.2]{Silverman} that $E[(\bar \omega)] = \overline{E[(\omega)]}$, where the bar symbol denotes complex conjugation. Hence, the right hand sides of \eqref{cm-identity} and \eqref{cm-identity-bar} are complex conjugates of each other, and adding \eqref{cm-identity} and \eqref{cm-identity-bar} shows that $e_2^*(\Lambda) \in \Q(j(\tau))$.

We now wish to find an explicit positive integer $M$ such that $M s_2(\tau) \in \mathcal{O}_L$. Firstly, choose a model of $E$ so that $e_4(\Lambda), e_6(\Lambda) \in \mathcal{O}_L$. Then $E$ has a model of the form
\begin{equation}
    y^2 = x^3 + A x + B
\end{equation}
where
\begin{equation*}
    A = -g_2(\Lambda)/4 \text{ and } B = -g_3(\Lambda)/4
\end{equation*}
both lie in $\mathcal{O}_L$. Suppose there exists a positive integer $M_1$ such that
\begin{equation*}
  M_1 \omega (\omega - \bar \omega) e_2^*(\Lambda) \in \mathcal{O}_L.
\end{equation*}
Then
\begin{equation*}
   M_2 e_2^*(\Lambda) \in \mathcal{O}_L,
\end{equation*}
where $M_2 = M_1 N_{K/\Q}(\omega (\omega - \bar \omega)) = M_1 acd$. Choose a positive integers $M_6$ so that
$M_6/e_6(\Lambda) \in \mathcal{O}_L$. The required positive integer $M$ is then given by $M = 7 M_2 M_6$ using the expression in \eqref{s2-invariance}

To determine a positive integer $M_1$, we note from \eqref{cm-identity} that it suffices to find an $M_1$ such that $M_1 (G_2(r/\omega,\Lambda) - e_2(\Lambda)) \in \mathcal{O}_L$ for all $r \in \Lambda/\omega \Lambda - \left\{ 0 \right\}$. But each $G_2(r/\omega,\Lambda) - e_2(\Lambda)$ is just the $x$-coordinate of a non-zero point in $E[(\omega)] \subseteq E[N]$ where $N = N_{K/\Q}(\omega)$ and 
\begin{equation*}
  E[N] = \left\{ P \in E : N P = 0 \right\} 
\end{equation*}
denotes the subgroup of $N$-torsion points of $E$.

Recall now that the (univariate) $N$-division polynomial $\psi_{N}$ for $E$ lies in $\Z[A,B,x] = \mathcal{O}_L[x]$ and its leading coefficient in $x$ is $N^2$  \cite[p.105]{Silverman}. The roots of $\psi_N$ are precisely the $x$-coordinates of $N$-torsion points of $E$. Hence, taking $M_1 = N^2$ suffices.

\end{proof}

\begin{lemma}
\label{discrete}
    Let $L$ be a number field with ring of integers $\mathcal{O}_L$ and denote by $\iota : L \rightarrow L_\R$ be the embedding of $L$ into its Minkowski space $L_\R$. Suppose $x_1, x_2 \in \mathcal{O}_L$ and
    \begin{equation*}
      ||\iota(x_1) - \iota(x_2)|| < \sqrt{n}
    \end{equation*}
    Then $x_1 = x_2$.
\end{lemma}
\begin{proof}
Recall that $\Gamma = \iota(\mathcal{O}_L)$ is a complete lattice in $L_\R$. Assume $x_1 \not= x_2$ so that $x = x_1 - x_2 \in \Gamma$ is non-zero and $|N_{L/\Q}(x)|$ is a positive integer. By the AM-GM inequality, we have that
\begin{equation*}
  ||\iota(x)||^2 = \sum_\tau |\tau(x)|^2 \ge n \left( \prod_{\tau} |\tau(x)|^2 \right)^{1/n} = n |N_{L/\Q}(x)|^{2/n} \ge n,
\end{equation*}
where $\tau$ runs through all embeddings of $L$. This contradicts our hypothesis on $\iota(x) = \iota(x_1) - \iota(x_2)$, hence $x_1 = x_2$.
\end{proof}
To prove $s_2(\tau)$ is equal to a candidate value $\alpha \in L$, we can apply Lemma~\ref{discrete} to check $M s_2(\tau)$ is equal to $M \alpha$, where $M$ is a choice of positive integer from Theorem~\ref{height}. Note to apply Lemma~\ref{discrete}, we need to be able to compute the conjugates of $s_2(\tau)$ numerically. This can be done because these conjugates are obtained by $s_2(\tau')$ for $\tau'$ corresponding to other elements in the class group of $\Q(\tau)$.

As an example, consider 
\begin{align*}
  \tau  & = \frac{-3 + \sqrt{-267}}{6}, \tau' = \frac{-1 + \sqrt{-267}}{2} \\
    J & = j(\tau)/1728 \text{ or } j(\tau')/1728 \\
  E_{J} : y^2 & = 4 x^3 - g_2(\Lambda) x - g_3(\Lambda) \\
  g_2(\Lambda) & = \frac{27 J}{J-1} D^2 = 864000 (113507917165866125 \mp 385015774793 \sqrt{89}) \\
  g_3(\Lambda) & = \frac{27 J}{J-1} D^3 = 52072539428448000 (113507917165866125 \mp 385015774793 \sqrt{89}) \\
  D & = 7 \cdot 11 \cdot 71 \cdot 167 \cdot 251 \cdot 263.
\end{align*}
Note that $j(\tau), j(\tau') \in L = \Q(\sqrt{89})$ are the two roots of the Hilbert class polynomial for discriminant $-267$, so that $s_2(\tau), s_2(\tau') \in L$ are the two distinct conjugates over $\Q$.

Computing $s_2(\tau), s_2(\tau')$ to high precision, we may use the LLL algorithm to guess an algebraic number of degree $2$ which is numerically close to $s_2(\tau),s_2(\tau')$, which in this instance yields
\begin{align}
    \label{s2-example}
    s_2(\tau) & \approx \frac{4110014282640 - 66461074000 \sqrt{89}}{5363953714273} = x_2 \\
    \notag s_2(\tau') & \approx \frac{4110014282640 + 66461074000 \sqrt{89}}{5363953714273} = x_2'.
\end{align}
To rigorously prove the $\approx$ in \eqref{s2-example} are $=$, it suffices to show that
\begin{align}
\label{minkowski-bound}
    & |M s_2(\tau) - M x_2| < 1 \\
    \notag & |M' s_2(\tau') - M' x_2'| < 1
\end{align}
by Lemma~\ref{discrete}, where $M$ is as in Theorem~\ref{height} for $\tau$, and $M'$ is the $M$ for $\tau'$. The positive integer $M$ for $\tau$ and $\tau'$, respectively, are obtained as follows,
\begin{align*}
    N & = 3 \cdot 23 \text{ (for $\tau$)}, 67 \text{ (for $\tau'$)} \\
    M_1 & = N^2 \\
    M_2 & = 267 M_1 N \\
    M_6 & = N_{L/\mathbb{Q}}(e_6(\Lambda)) = 2^{20} \cdot 3^6 \cdot 5^4 \cdot 11^7 \cdot 17^3 \cdot 47^3 \cdot 71^4 \cdot 167^4 \cdot 251^4 \cdot 263^4 \\
    M & = 7 M_2 M_6,
\end{align*}
and have approximately $72$ decimal digits so \eqref{minkowski-bound} can be readily verified.

For the interested reader, we detail below a computationally efficient method to numerically compute $s_2(\tau)$. We explain the method in the above example and show how it can compute $s_2(\tau)$ to 16,000 decimal digits of accuracy for instance.

For $\tau = \frac{-3 + \sqrt{-267}}{6}$, we have that $q = \exp(2 \pi i \tau) = -\exp\left(-\pi \sqrt{\frac{89}{3}}\right)$. Further, we know that
\begin{equation*}
    E_2(\tau) = 1 - 24 \sum_{n = 1}^\infty \frac{n}{q^{-n} - 1}.
\end{equation*}
For $k \in \mathbb{N}$, define
\begin{equation*}
    E_2(\tau, k) = 1 - 24 \sum_{n = 1}^{k-1} \frac{n}{q^{-n} - 1}.
\end{equation*}
Then, subtracting the above equations, we obtain
\begin{equation*}
    E_2(\tau, k) - E_2(\tau) = 24 \sum_{n = k}^\infty \frac{n}{q^{-n} - 1},
\end{equation*}
which for $\text{Im}(\tau) \ge \frac{\log 2}{2 \pi}$ implies
\begin{equation*}
    |E_2(\tau, k) - E_2(\tau)| \le 2 \cdot 24 \sum_{n = k}^\infty n r^n \le 2 \cdot 24 \cdot 3 k \frac{r^k}{(1-r)^2} \le 576 k r^k,
\end{equation*}
where $r = |q|$. A similar analysis for $E_4(\tau)$ and $E_6(\tau)$ yields
\begin{equation*}
    |E_4(\tau, k) - E_4(\tau)| \le 2 \cdot 240 \sum_{n = k}^\infty n^3 r^n \le 2 \cdot 240 \cdot 32 k^3 \frac{r^k}{(1-r)^4} \le 245760 k^3 r^k,
\end{equation*}
and
\begin{equation*}
    |E_6(\tau, k) - E_6(\tau)| \le 2 \cdot 504 \sum_{n = k}^\infty n^5 r^n \le 2 \cdot 504 \cdot 522 k^5 \frac{r^k}{(1-r)^6} \le 33675264 k^5 r^k.
\end{equation*}
Note that in this case, $r = |q| \approx 3.3076 \times 10^{-8}$. Setting $k = 3000$ and using {\tt Mathematica} to obtain approximations $\widetilde{E_{2l}}(\tau, k)$ (accurate up to 18,000 digits) of $E_{2l}(\tau, k)$, we get
\begin{equation*}
    |\widetilde{E_{2l}}(\tau, 3000) - E_{2l}(\tau, 3000)| \le 10^{-18000},
\end{equation*}
which implies
\begin{equation*}
    |\widetilde{E_{2l}}(\tau, 3000) - E_{2l}(\tau)| \le 2 \times 10^{-18000}.
\end{equation*}
Now, we use these approximate values in the definition of $s_2(\tau)$ to obtain an approximation $\widetilde{s_2}(\tau)$ of $s_2(\tau)$ (accurate up to 17,000 digits assuming $\widetilde{E_{2l}}(\tau, 3000)$ to be the true values). It can be easily shown using basic error analysis that this estimate is correct up to 16,000 digits, i.e.
\begin{equation*}
    |\widetilde{s_2}(\tau) - s_2(\tau)| \le 10^{-16000}.
\end{equation*}
It takes approximately 39 minutes to establish 16,000 decimal digits of $s_2(\tau)$ in {\tt Mathematica}.

It is known from the theory of elliptic curves with complex multiplication that $j(\tau)$ and $J(\tau)$ are rational for $\tau = \sqrt{-N}$ iff $N = 1, 2, 3, 4, 7$ and also for $\tau = \frac{-1 + \sqrt{-N}}{2}$ if and only if $N = 3, 7, 11, 19, 27, 43, 67, 163$. More generally, the class number problem has been solved explicitly in \cite{Watkins} for imaginary quadratic fields of class number $\le 100$. The largest (in absolute value) discriminant of an imaginary quadratic field with class number $\le 2$ is $-427$. Using the class number formula for orders, it is still the case that the largest discriminant (in absolute value) of an imaginary quadratic order with class number $\le 2$ is $-427$. Thus, the class polynomials can be enumerated for all imaginary quadratic orders of class number $\le 2$ using known algorithms.

In what follows, we note that the rational value may not be attained by the branch of uniformizer we have selected. In such a case, we have translated $\tau$ by a suitable coset representative for the congruence subgroup in $\SL_2(\Z)$ so that the branch we have chosen attains the rational value.

\subsection{Case 1B}

Using the modular relation for this case and rational values of $J$, we can determine which singular values of $J$ give rise to a rational value of $s$. Table~\ref{tab:rational-s} gives rational values of $s(\tau)$ and $s_2(\tau)$ for $\tau$ in the fundamental domain of $\Gamma_{1B}$.

\begin{table}[H]
\centering
\renewcommand{\arraystretch}{1.5}
\begin{tabular}{ccc}
\hline
$\tau$ & $s(\tau)$ & $s_2(\tau)$ \\
\hline
$\sqrt{-1}$ & $\frac{1}{2}$ & - \\
\hline
\end{tabular}
\caption{Special values of $s$ and $s_2$}
\label{tab:rational-s}
\end{table}

\subsection{Case 2B}

Using the modular relation for this case and rational values of $J$, we can determine which singular values of $J$ give rise to a rational value of $t$. Table~\ref{tab:rational-t} gives the rational values of $t(\tau)$ and $s_2(\tau)$ for $\tau$ in the fundamental domain of $\Gamma_0(2)$.

\begin{table}[H]
\centering
\renewcommand{\arraystretch}{1.5}
\begin{tabular}{ccc}
\hline
$\tau$ & $t(\tau)$ & $s_2(\tau)$ \\
\hline
$\sqrt{-1}$ & $-8$ & - \\
$\frac{\sqrt{-1}}{2}$ & $-\frac{1}{8}$ & $\frac{11}{21}$ \\
$\frac{\sqrt{-2}}{2}$ & $-1$ & $\frac{5}{14}$ \\
$\frac{1+\sqrt{-3}}{4}$ & $\frac{1}{4}$ & $\frac{5}{11}$ \\
$\frac{1+\sqrt{-7}}{8}$ & $\frac{1}{64}$ & $\frac{85}{133}$ \\
$\frac{-1+\sqrt{-3}}{2}$ & $4$ & $0$ \\
$\frac{-1+\sqrt{-7}}{2}$ & $64$ & $\frac{5}{21}$ \\
\hline
\end{tabular}
\caption{Special values of $t$ and $s_2$}
\label{tab:rational-t}
\end{table}

\subsection{Case 2C}

Using the modular relation for this case and rational values of $J$, we can determine which singular values of $J$ give rise to a rational value of $\lambda$. Table~\ref{tab:rational-lambda} gives the rational values of $\lambda(\tau)$ and $s_2(\tau)$ for $\tau$ in the fundamental domain of $\Gamma(2)$.

\begin{table}[H]
\centering
\renewcommand{\arraystretch}{1.5}
\begin{tabular}{ccc}
\hline
$\tau$ & $\lambda(\tau)$ & $s_2(\tau)$ \\
\hline
$\sqrt{-1}$ & $\frac{1}{2}$ & - \\
$1+\sqrt{-1}$ & $-1$ & - \\
$\frac{1+\sqrt{-1}}{2}$ & $2$ & - \\
\hline
\end{tabular}
\caption{Special values of $\lambda$ and $s_2$}
\label{tab:rational-lambda}
\end{table}

\subsection{Case 3B}

Using the modular relation for this case and rational values of $J$, we can determine which singular values of $J$ give rise to a rational value of $u$. Table~\ref{tab:rational-u} gives the rational values of $u(\tau)$ and $s_2(\tau)$ for $\tau$ in the fundamental domain of $\Gamma_0(3)$.

\begin{table}[H]
\centering
\renewcommand{\arraystretch}{1.5}
\begin{tabular}{ccc}
\hline
$\tau$ & $u(\tau)$ & $s_2(\tau)$ \\
\hline
$\frac{\sqrt{-3}}{3}$ & $-1$ & $\frac{5}{11}$ \\
$\frac{-1+\sqrt{-3}}{2}$ & $9$ & $0$ \\
$\frac{1+\sqrt{-3}}{6}$ & $\frac{1}{9}$ & $\frac{160}{253}$ \\
\hline
\end{tabular}
\caption{Special values of $u$ and $s_2$}
\label{tab:rational-u}
\end{table}

\subsection{Case 2A}

Using the modular relations for cases 2A and 2B, and algebraic values of $J$ (of degree $\le 2$), we can determine which singular values of $J$ give rise to rational values of $v$. Table~\ref{tab:rational-v} gives rational values of $v(\tau)$ and the corresponding values of $t(\tau)$ and $s_2(\tau)$ for $\tau$ in the fundamental domain of $\Gamma_0^+(2)$.

\begin{table}[H]
\centering
\renewcommand{\arraystretch}{1.5}
\begin{tabular}{cccc}
\hline
$\tau$ & $v(\tau)$ & $t(\tau)$ & $s_2(\tau)$\\
\hline
$\frac{-1+\sqrt{-7}}{4}$ & $81/256$ & $\frac{47 - 45 \sqrt{-7}}{128}$ & $\frac{5}{21}$ \\
$\frac{-1+\sqrt{-3}}{2}$ & $-9/16$ & $4$ & $0$ \\
$\frac{-1+\sqrt{-5}}{2}$ & $-4$ & $9+4\sqrt{5}$ & $\frac{139 - 45 \sqrt{5}}{418}$ \\
$\frac{-1+\sqrt{-7}}{2}$ & $-3969/256$ & $64$ & $\frac{5}{21}$ \\
$\frac{-1+\sqrt{-9}}{2}$ & $-48$ & $97 + 56 \sqrt{3}$ & $\frac{79 - 15 \sqrt{3}}{154}$ \\
$\frac{-1+\sqrt{-13}}{2}$ & $-324$ & $649 + 180 \sqrt{13}$ & $\frac{2015 - 125 \sqrt{13}}{3354}$ \\
$\frac{-1+\sqrt{-25}}{2}$ & $-25920$ & $51841 + 23184 \sqrt{5}$ & $\frac{6263 - 375 \sqrt{5}}{8778}$ \\
$\frac{-1+\sqrt{-37}}{2}$ & $-777924$ & $1555849 + 255780 \sqrt{37}$ & $\frac{121707985 - 2054375 \sqrt{37}}{159196422}$ \\
$\frac{\sqrt{-6}}{2}$ & $9$ & $-17 - 12 \sqrt{2}$ & $\frac{21 - 5 \sqrt{2}}{46}$ \\
$\frac{\sqrt{-10}}{2}$ & $81$ & $-161 - 72 \sqrt{5}$ & $\frac{103 - 12 \sqrt{5}}{186}$ \\
$\frac{\sqrt{-18}}{2}$ & $2401$ & $-4801 - 1960 \sqrt{6}$ & $\frac{712075 - 49230 \sqrt{6}}{1074514}$ \\
$\frac{\sqrt{-22}}{2}$ & $9801$ & $-19601 - 13860 \sqrt{2}$ & $\frac{25355 - 2625 \sqrt{2}}{36498}$ \\
$\frac{\sqrt{-58}}{2}$ & $96059601$ & $-192119201 - 35675640 \sqrt{29}$ & $\frac{8424836255 - 120803060 \sqrt{29}}{10376469642}$ \\
$\sqrt{-1}$ & $81/32$ & $-8$ & - \\
\hline
\end{tabular}
\caption{Special values of $v$, $t$, and $s_2$}
\label{tab:rational-v}
\end{table}

\subsection{Case 3A}

Using the modular relations for cases 3A and 3B, and algebraic values of $J$ (of degree $\le 2$), we can determine which singular values of $J$ give rise to rational values of $w$. Table~\ref{tab:rational-w} gives rational values of $w(\tau)$ and the corresponding values of $u(\tau)$ and $s_2(\tau)$ for $\tau$ in the fundamental domain of $\Gamma_0^+(3)$.

\begin{table}[H]
\centering
\renewcommand{\arraystretch}{1.5}
\begin{tabular}{cccc}
\hline
$\tau$ & $w(\tau)$ & $u(\tau)$ & $s_2(\tau)$\\
\hline
$\frac{-1+\sqrt{-3}}{2}$ & $-16/9$ & $9$ & $0$ \\
$\frac{-1+\sqrt{-2}}{3}$ & $2/27$ & $\frac{23 - 10 \sqrt{-2}}{27}$ & $\frac{5}{14}$ \\
$\frac{-1+\sqrt{-11}}{6}$ & $16/27$ & $-\frac{5 + 8 \sqrt{-11}}{27}$ & $\frac{32}{77}$ \\
$\frac{-3+\sqrt{-15}}{6}$ & $-1/4$ & $\frac{3 + \sqrt{5}}{2}$ & $\frac{ 21 - 8 \sqrt{5}}{77}$ \\
$\frac{-3+\sqrt{-51}}{6}$ & $-16$ & $33 + 8 \sqrt{17}$ & $\frac{2448 - 400 \sqrt{17}}{5593}$ \\
$\frac{-3+\sqrt{-75}}{6}$ & $-80$ & $161 + 72 \sqrt{5}$ & $\frac{25184 - 4800 \sqrt{5}}{46079}$ \\
$\frac{-3+\sqrt{-123}}{6}$ & $-1024$ & $2049 + 320 \sqrt{41}$ & $\frac{27256800 - 1130560 \sqrt{41}}{41672113}$ \\
$\frac{-3+\sqrt{-147}}{6}$ & $-3024$ & $6049 + 1320 \sqrt{21}$ & $\frac{3847200 - 193920 \sqrt{21}}{5621341}$ \\
$\frac{-3+\sqrt{-267}}{6}$ & $-250000$ & $500001 + 53000 \sqrt{89}$ & $\frac{4110014282640 - 66461074000 \sqrt{89}}{5363953714273}$ \\
$\frac{\sqrt{-6}}{3}$ & $2$ & $-3 - 2 \sqrt{2}$ & $\frac{21 - 5 \sqrt{2}}{46}$ \\
$\frac{\sqrt{-12}}{3}$ & $27/2$ & $-26 - 15 \sqrt{3}$ & $\frac{6015 - 1500 \sqrt{3}}{11891}$ \\
$\frac{\sqrt{-15}}{3}$ & $125/4$ & $-\frac{123 + 55 \sqrt{5}}{2}$ & $\frac{2439 - 440 \sqrt{5}}{4543}$ \\
\hline
\end{tabular}
\caption{Special values of $w$, $u$, and $s_2$}
\label{tab:rational-w}
\end{table}

\begin{remark}
The entries marked - in the $s_2(\tau)$ column in Tables~\ref{tab:rational-s}--\ref{tab:rational-w} correspond to the points at which $s_2(\tau)$ is not well defined, or equivalently $E_6(\tau) = 0$. We use the alternate form \eqref{Chudnovsky-Ramanujan-2} of precursor formula to obtain Chudnovsky-Ramanujan type formulae at points where $s_2(\tau) = 0$ or undefined. The values of $E_2^*(\tau)$ (which are given with proof in the discussion below) are required to obtain the final formulae at such $\tau$.
\end{remark}

It is known that $E_2(\tau)$ is a quasi-modular form that satisfies \cite[Section 2.3]{Zagier}
\begin{equation}
\label{E_2-quasi}
    E_2\left(\frac{a\tau + b}{c\tau+d}\right) = (c\tau + d)^2 E_2(\tau) - \frac{6i}{\pi} c (c\tau+d)
\end{equation}
for $\tau \in \mathfrak{H}$ and $g = \begin{pmatrix} a & b \\ c & d \end{pmatrix} \in SL_2(\mathbb{Z})$. 
Setting $g = \begin{pmatrix} 1 & 1 \\ 0 & 1 \end{pmatrix}$ in \eqref{E_2-quasi}, we obtain 
\begin{equation}
\label{eqn:E_2:tau+1}
    E_2(\tau+1) = E_2(\tau) 
\end{equation}
for $\tau \in \mathfrak{H}$. And setting $g = \begin{pmatrix} 0 & -1 \\ 1 & 0 \end{pmatrix}$ in \eqref{E_2-quasi}, we obtain
\begin{equation}
\label{eqn:E_2:-1/tau}
    E_2\left(-\frac{1}{\tau}\right) = \tau^2 E_2(\tau) - \frac{6i}{\pi} \tau 
\end{equation}
for $\tau \in \mathfrak{H}$. We will now derive some special values of $E_2(\tau)$ and $E_2^*(\tau)$ using the above identities which we will need later. Setting $\tau = i$ in \eqref{eqn:E_2:-1/tau} gives $E_2(i) = 3/\pi$. Therefore, using \eqref{eqn:E_2:tau+1}, we have $E_2(i + 1) = 3/\pi$. Setting $\tau = \rho$ in \eqref{eqn:E_2:-1/tau} and using \eqref{eqn:E_2:tau+1} gives $E_2(\rho) = 2\sqrt{3}/\pi$ and $E_2(\rho+1) = 2\sqrt{3}/\pi$. Using these values of $E_2(\tau)$ in the definition of $E_2^*(\tau)$, we find that $E_2^*(\tau) = 0$ for $\tau = i, i+1, \rho, \rho+1$.

\section{Examples}

The derivative of the hypergeometric function with $z$ in its disc of convergence is given by
\begin{equation*}
  \frac{d}{dz} \vphantom{1em}_2 F_1 \left(a, b ; c; z \right) = \frac{ab}{c} \vphantom{1em}_2 F_1 \left(a+1, b+1 ; c+1; z \right).
\end{equation*}
Therefore, the derivative term in the precursor formula can be written as
\begin{equation*}
  \frac{d}{d\xi} \vphantom{1em}_2 F_1 \left(a, b ; c; \nu(\xi) \right)^2 = 2 \left[ \frac{d}{d\xi} \nu(\xi) \right] \vphantom{1em}_2 F_1 \left(a, b ; c; \nu(\xi) \right) \frac{ab}{c} \vphantom{1em}_2 F_1 \left(a+1, b+1 ; c+1; \nu(\xi) \right),
\end{equation*}
where $\nu(\xi)$ is one of the six possible expressions. In this section, we derive rational series for $1/\pi$ of the form
\begin{equation*}
    \sum_{n=0}^\infty (An + B) s(n) C^{-n} = \frac{D}{\pi},
\end{equation*}
where $s(n) \in \mathbb{Z}$ for each $n \ge 0$, $A, B \in \mathbb{Z}$ with $(A, B) = 1$, $C \in \mathbb{Z} \setminus \{0\}$, and $D \in \overline{\mathbb{Q}}$.

\subsection{Case 1B}

For $\tau \in C_{s,0}$, we have
\begin{align*}
    F = \vphantom{1em}_2 F_1 \left( \frac{1}{6}, \frac{5}{6} ; 1; s \right)^2 &= \left[\sum_{n=0}^\infty {6n \choose 3n} {3n \choose n} \left(\frac{s}{432}\right)^n\right]^2\\
    &= \sum_{n=0}^\infty \sum_{k=0}^n {6k \choose 3k} {3k \choose k} {6n-6k \choose 3n-3k} {3n-3k \choose n-k} \left(\frac{s}{432}\right)^n.
\end{align*}
Using the above identity, Theorem~\ref{thm:s=0} reduces to
\begin{equation*}
    \sum_{n=0}^\infty \left(\frac{1-2s}{6} + \frac{2s-1}{6} s_2(\tau) + n (1-s)\right) \sum_{k=0}^n {6k \choose 3k} {3k \choose k} {6n-6k \choose 3n-3k} {3n-3k \choose n-k} \left(\frac{s}{432}\right)^n = \frac{a \tau^2}{\pi \sqrt{d}} - \frac{\tau i}{\pi}.
\end{equation*}
The value $\tau = \sqrt{-1}$ is such that $s(\tau)$ is rational, is as in \eqref{Chudnovsky}, and lies in $C_{s,0}$. Thus, the above equation holds for this value of $\tau$. We state the identity obtained in Table~\ref{tab:convolutional-double}.

\begin{lemma}
\label{lem:wz-0}
Let $n$ be a non-negative integer. Then, we have
\begin{equation}
\label{binomial-0}
    (-432)^n \sum_{k=0}^n \left(\frac{(1/6)_k (5/6)_{n-k}}{(1)_k (1)_{n-k}}\right)^2 = \sum_{k=0}^n {6k \choose 3k} {3k \choose 2k} {2k \choose k} {n+k \choose n-k} (-432)^{n-k}.
\end{equation}
\end{lemma}

\begin{proof}
For $n = 0$ and $n = 1$, the result is true since both sides of \eqref{binomial-0} are equal to $1$ and $-312$ respectively. Further, both sides of \eqref{binomial-0} satisfy the recurrence relation
\begin{equation*}
    (n+2)^3 a(n+2) + 24(2n+3)(18n^2+54n+49) a(n+1) + 186624(n+1)^3 a(n) = 0,
\end{equation*}
which can be verified using Zeilberger's algorithm as follows: For $k \in \{0, 1, \dots, n\}$, define
\begin{align*}
    F_1(n, k) &= \left(\frac{(1/6)_k (5/6)_{n-k}}{(1)_k (1)_{n-k}}\right)^2 (-432)^n\\
    F_2(n, k) &= {6k \choose 3k} {3k \choose 2k} {2k \choose k} {n+k \choose n-k} (-432)^{n-k},
\end{align*}
which are hypergeometric terms in both variables $n,k$. We find rational functions $R_i(n,k)$ so that
\begin{equation*}
\begin{split}
    (n+2)^3 F_i(n+2, k) + 24(2n+3)(18n^2+54n+49) F_i(n+1, k) + 186624 (n+1)^3 F_i(n, k) \\
    = R_i(n,k+1) F_i(n,k+1) - R_i(n,k) F_i(n,k),
\end{split}
\end{equation*}
and then sum over $k \in \{0, 1, \dots, n\}$. The proof certificates are:
\begin{align*}
    R_1(n,k) & = -\frac{144k^2(6n-6k+5)^2(12n^2-36kn+51n+24k^2-74k+54)}{(n-k+1)^2(n-k+2)^2}, \\
    R_2(n,k) & = -\frac{373248 k^4(2n+3)}{(n-k+1)(n-k+2)}.
\end{align*}
\end{proof}

For $\tau \in C_{s/(s-1), 0}$, we have
\begin{align*}
    \vphantom{1em}_2 F_1 \left( \frac{1}{6}, \frac{1}{6} ; 1; \frac{s}{s-1} \right)^2 &= \left(1-\frac{s}{s-1}\right)^{2/3} \vphantom{1em}_2 F_1 \left( \frac{1}{6}, \frac{1}{6} ; 1; \frac{s}{s-1} \right) \vphantom{1em}_2 F_1 \left( \frac{5}{6}, \frac{5}{6} ; 1; \frac{s}{s-1} \right) \\
    &= (1-s)^{-2/3} \sum_{n=0}^\infty \sum_{k=0}^n \left(\frac{(1/6)_k (5/6)_{n-k}}{(1)_k (1)_{n-k}}\right)^2 \left(\frac{s}{s-1}\right)^n \\
    &= (1-s)^{-2/3} \sum_{n=0}^\infty \sum_{k=0}^n {6k \choose 3k} {3k \choose 2k} {2k \choose k} {n+k \choose n-k} (-432)^{-k} \left(\frac{s}{s-1}\right)^n,
\end{align*}
where the first equality comes from Euler's transformation and the third from Lemma~\ref{lem:wz-0}. Using the above identity, Theorem~\ref{thm-1:s=0} reduces to
\begin{equation*}
    \sum_{n=0}^\infty \left(\frac{1}{6} + \frac{2s}{3} + \frac{2s-1}{6} s_2(\tau) + n\right) \sum_{k=0}^n {6k \choose 3k} {3k \choose 2k} {2k \choose k} {n+k \choose n-k} (-432)^{-k} \left(\frac{s}{s-1}\right)^n = \left[\frac{a \tau^2}{\pi \sqrt{d}} - \frac{\tau i}{\pi}\right] (1-s).
\end{equation*}

\subsection{Case 2B}

Setting $a = \frac{1}{4}, b = \frac{1}{4}$ in Clausen's formula \cite[p. 116]{Andrews}, we obtain
\begin{equation}
\label{eqn:2B-Clausen}
    \vphantom{1em}_2 F_1 \left( \frac{1}{4}, \frac{1}{4} ; 1; \frac{1}{t} \right)^2 = \vphantom{1em}_3 F_2 \left( \frac{1}{2}, \frac{1}{2}, \frac{1}{2} ; 1, 1 ; \frac{1}{t} \right).
\end{equation}
Using the above identity, Theorem~\ref{thm:t=infty} reduces to
\begin{equation}
\label{eqn:t=infty}
    \sum_{n=0}^{\infty} \left(\frac{1}{6} - \frac{1}{6} \left(\frac{t+8}{t-4}\right) s_2(\tau) + n\right) \frac{\left(\frac{1}{2}\right)_n^3}{n!^3} t^{-n} = \frac{a}{\pi \sqrt{d}}\left(\frac{t}{t-1}\right)^{1/2}.
\end{equation}
The values $\tau = \frac{\sqrt{-2}}{2}, \sqrt{-1}, \frac{-1+\sqrt{-3}}{2}, \frac{-1+\sqrt{-7}}{2}$ are such that $t(\tau)$ is rational, are as in \eqref{Chudnovsky}, and lie in $C_{1/t,\infty}$. Thus the above equation holds for these values of $\tau$. We state all possible identities in Table~\ref{tab:single}.

For $\tau \in C_{1/(1-t),\infty}$, we have
\begin{align*}
    F = \vphantom{1em}_2 F_1 \left( \frac{1}{4}, \frac{3}{4} ; 1; \frac{1}{1-t} \right)^2 &= \left[\sum_{n=0}^\infty {4n \choose 2n} {2n \choose n} \left(\frac{1}{64(1-t)}\right)^n\right]^2\\
    &= \sum_{n=0}^\infty \sum_{k=0}^n {4k \choose 2k} {2k \choose k} {4n-4k \choose 2n-2k} {2n-2k \choose n-k} \left(\frac{1}{64(1-t)}\right)^n.
\end{align*}
Using the above identity, Theorem~\ref{thm-1:t=infty} reduces to
\begin{equation*}
    \sum_{n=0}^\infty \left(\frac{t+2}{6} - \frac{(t-1)(t+8)}{6(t-4)} s_2(\tau) + n t\right) \sum_{k=0}^n {4k \choose 2k} {2k \choose k} {4n-4k \choose 2n-2k} {2n-2k \choose n-k} \left(\frac{1}{64(1-t)}\right)^n = \frac{a}{\pi \sqrt{d}} (t-1).
\end{equation*}
The values $\tau = \sqrt{-1}, \frac{\sqrt{-1}}{2}, \frac{\sqrt{-2}}{2}, \frac{-1+\sqrt{-3}}{2}, \frac{-1+\sqrt{-7}}{2}$ are such that $t(\tau)$ is rational, are as in \eqref{Chudnovsky}, and lie in $C_{1/(1-t),\infty}$. Thus the above equation holds for these values of $\tau$. We state all possible identities in Table~\ref{tab:convolutional-double}.

\begin{lemma}
\label{lem:wz-1}
Let $n$ be a non-negative integer. Then, we have
\begin{equation}
\label{binomial-1}
    (-64)^n \sum_{k=0}^n \left(\frac{(1/4)_k (3/4)_{n-k}}{(1)_k (1)_{n-k}}\right)^2 = \sum_{k=0}^n {4k \choose 2k} {2k \choose k}^2 {n+k \choose n-k} (-64)^{n-k}.
\end{equation}
\end{lemma}

\begin{proof}
For $n = 0$ and $n = 1$, the result is true since both sides of \eqref{binomial-1} are equal to $1$ and $-40$ respectively. Further, both sides of \eqref{binomial-1} satisfy the recurrence relation
\begin{equation*}
    (n+2)^3 a(n+2) + 8(2n+3)(8n^2+24n+21) a(n+1) + 4096 (n+1)^3 a(n) = 0,
\end{equation*}
which can be verified using Zeilberger's algorithm as follows: For $k \in \{0, 1, \dots, n\}$, define
\begin{align*}
    F_1(n, k) &= \left(\frac{(1/4)_k (3/4)_{n-k}}{(1)_k (1)_{n-k}}\right)^2 (-64)^n\\
    F_2(n, k) &= {4k \choose 2k} {2k \choose k}^2 {n+k \choose n-k} (-64)^{n-k},
\end{align*}
which are hypergeometric terms in both variables $n,k$. We find rational functions $R_i(n,k)$ so that
\begin{equation*}
\begin{split}
    (n+2)^3 F_i(n+2, k) + 8(2n+3)(8n^2+24n+21) F_i(n+1, k) + 4096 (n+1)^3 F_i(n, k) \\
    = R_i(n,k+1) F_i(n,k+1) - R_i(n,k) F_i(n,k),
\end{split}
\end{equation*}
and then sum over $k \in \{0, 1, \dots, n\}$. The proof certificates are:
\begin{align*}
    R_1(n,k) & = \frac{16k^2(4n-4k+3)^2(8kn-3n-8k^2+18k-6)}{(n-k+1)^2(n-k+2)^2}, \\
    R_2(n,k) & = -\frac{8192k^4(2n+3)}{(n-k+1)(n-k+2)}.
\end{align*}
\end{proof}

For $\tau \in C_{1/t, \infty}$, we have
\begin{align*}
    \vphantom{1em}_2 F_1 \left( \frac{1}{4}, \frac{1}{4} ; 1; \frac{1}{t} \right)^2 &= (1-1/t)^{1/2} \vphantom{1em}_2 F_1 \left( \frac{1}{4}, \frac{1}{4} ; 1; \frac{1}{t} \right) \vphantom{1em}_2 F_1 \left( \frac{3}{4}, \frac{3}{4} ; 1; \frac{1}{t} \right) \\
    &= (1-1/t)^{1/2} \sum_{n=0}^\infty \sum_{k=0}^n \left(\frac{(1/4)_k (3/4)_{n-k}}{(1)_k (1)_{n-k}}\right)^2 \left(\frac{1}{t}\right)^n \\
    &= (1-1/t)^{1/2} \sum_{n=0}^\infty \sum_{k=0}^n {4k \choose 2k} {2k \choose k}^2 {n+k \choose n-k} (-64)^{-k} \left(\frac{1}{t}\right)^n,
\end{align*}
where the first equality comes from Euler's transformation and the third from Lemma~\ref{lem:wz-1}. Using the above identity, Theorem~\ref{thm:t=infty} reduces to
\begin{equation*}
    \sum_{n=0}^\infty \left(-1 + \frac{t+2}{6} - \frac{(t-1)(t+8)}{6(t-4)} s_2(\tau) + n (t-1)\right) \sum_{k=0}^n {4k \choose 2k} {2k \choose k}^2 {n+k \choose n-k} (-64)^{-k} \left(\frac{1}{t}\right)^n = \frac{a}{\pi \sqrt{d}} t.
\end{equation*}
The values $\tau = \sqrt{-1}, \frac{-1+\sqrt{-3}}{2}, \frac{-1+\sqrt{-7}}{2}$ are such that $t(\tau)$ is rational, are as in \eqref{Chudnovsky}, and lie in $C_{1/t,\infty}$. Thus the above equation holds for these values of $\tau$. We state all possible identities in Table~\ref{tab:double}.

\subsection{Case 2C}

The values $\tau = \sqrt{-1}, 1+\sqrt{-1}$ are such that $\lambda(\tau)$ is rational, are as in \eqref{Chudnovsky}, and lie in $C_{\lambda,\infty}$. Thus Theorem~\ref{thm:lambda=0} holds for these values of $\tau$. We state both possible identities below.\\

$\tau = \sqrt{-1}$:
\begin{equation}
\label{2C-1}
    {}_{2}F_{1}\bigg(\frac{1}{2}, \frac{1}{2}; 1; \frac{1}{2}\bigg) {}_{2}F_{1}\bigg(\frac{3}{2}, \frac{3}{2}; 2; \frac{1}{2}\bigg) = \frac{8}{\pi}.
\end{equation}

$\tau = 1+\sqrt{-1}$:
\begin{equation}
\label{2C-2}
    {}_{2}F_{1}\bigg(\frac{1}{2}, \frac{1}{2}; 1; -1\bigg)^2 - {}_{2}F_{1}\bigg(\frac{1}{2}, \frac{1}{2}; 1; -1\bigg) {}_{2}F_{1}\bigg(\frac{3}{2}, \frac{3}{2}; 2; -1\bigg) = \frac{1}{\pi}.
\end{equation}

For $\tau \in C_{\lambda,\infty}$, we have
\begin{equation*}
    F = \vphantom{1em}_2 F_1 \left( \frac{1}{2}, \frac{1}{2} ; 1; \lambda \right)^2 = \sum_{n=0}^\infty \sum_{k=0}^n {2k \choose k}^2 {2n-2k \choose n-k}^2 \left(\frac{\lambda}{16}\right)^n.
\end{equation*}
Using the above identity, Theorem~\ref{thm:lambda=0} reduces to
\begin{equation*}
    \sum_{n=0}^\infty \left(\frac{1-2\lambda}{3} + \frac{(1-2\lambda)(\lambda+1)(\lambda-2)}{6(\lambda^2-\lambda+1)} s_2(\tau) + n(1-\lambda)\right) \sum_{k=0}^n {2k \choose k}^2 {2n-2k \choose n-k}^2 \left(\frac{\lambda}{16}\right)^n = \frac{2a}{\pi \sqrt{d}}.
\end{equation*}
The value $\tau = \sqrt{-1}$ is such that $t(\tau)$ is rational, is as in \eqref{Chudnovsky}, and lies in $C_{\lambda,\infty}$. Thus the above equation holds for this value of $\tau$. We state the identity obtained in Table~\ref{tab:convolutional-euler-double}.

\subsection{Case 3B}

The value $\tau = \frac{\sqrt{-3}}{3}, \frac{-1+\sqrt{-3}}{2}$ are such that $u(\tau)$ is rational, are as in \eqref{Chudnovsky}, and lie in $C_{1/u,\infty}$. Thus Theorem~\ref{thm:u=infty} holds for these value of $\tau$. We state both possible identities below.\\

$\tau = \frac{\sqrt{-3}}{3}$:
\begin{equation}
\label{3B-1}
    \frac{2}{3} \vphantom{1em}_2 F_1 \left( \frac{1}{3}, \frac{1}{3} ; 1; -1 \right)^2 - \frac{4}{9} \vphantom{1em}_2 F_1 \left( \frac{1}{3}, \frac{1}{3} ; 1; -1 \right) \vphantom{1em}_2 F_1 \left( \frac{4}{3}, \frac{4}{3} ; 2; -1 \right) = \frac{3^{1/2}}{2^{2/3} \pi}.
\end{equation}

$\tau = \frac{-1+\sqrt{-3}}{2}$:
\begin{equation}
\label{3B-2}
    \frac{4}{3} \vphantom{1em}_2 F_1 \left( \frac{1}{3}, \frac{1}{3} ; 1; \frac{1}{9} \right)^2 + \frac{16}{81} \vphantom{1em}_2 F_1 \left( \frac{1}{3}, \frac{1}{3} ; 1; \frac{1}{9} \right) \vphantom{1em}_2 F_1 \left( \frac{4}{3}, \frac{4}{3} ; 2; \frac{1}{9} \right) = \frac{3^{5/6}2}{\pi}.
\end{equation}

For $\tau \in C_{1/(1-u),\infty}$, we have
\begin{align*}
    \vphantom{1em}_2 F_1 \left( \frac{1}{3}, \frac{2}{3} ; 1; \frac{1}{1-u} \right)^2 &= \left[\sum_{n=0}^\infty {3n \choose n} {2n \choose n} \left(\frac{1}{27(1-u)}\right)^n\right]^2\\
    &= \sum_{n=0}^\infty \sum_{k=0}^n {3k \choose k} {2k \choose k} {3n-3k \choose n-k} {2n-2k \choose n-k} \left(\frac{1}{27(1-u)}\right)^n.
\end{align*}
Using the above identity, Theorem~\ref{thm-1:u=infty} reduces to
\begin{equation*}
    \sum_{n=0}^\infty \left(\frac{u+3}{6} - \frac{u^2+18u-27}{6(u-9)} s_2(\tau) + n u\right) \sum_{k=0}^n {3k \choose k} {2k \choose k} {3n-3k \choose n-k} {2n-2k \choose n-k} \left(\frac{1}{27(1-u)}\right)^n = \frac{a}{\pi \sqrt{d}} (u-1).
\end{equation*}
The values $\tau = \frac{\sqrt{-3}}{3}, \frac{-1+\sqrt{-3}}{2}$ are such that $u(\tau)$ is rational, are as in \eqref{Chudnovsky}, and lie in $C_{1/(1-u),\infty}$. Thus the above equation holds for these values of $\tau$. We state all possible identities in Table~\ref{tab:convolutional-double}.

\begin{lemma}
\label{lem:wz-2}
Let $n$ be a non-negative integer. Then, we have
\begin{equation}
\label{binomial-2}
    (-27)^n \sum_{k=0}^n \left(\frac{(1/3)_k (2/3)_{n-k}}{(1)_k (1)_{n-k}}\right)^2 = \sum_{k=0}^n {3k \choose k} {2k \choose k}^2 {n+k \choose n-k} (-27)^{n-k}.
\end{equation}
\end{lemma}

\begin{proof}
For $n = 0$ and $n = 1$, the result is true since both sides of \eqref{binomial-2} are equal to $1$ and $-15$ respectively. Further, both sides of \eqref{binomial-2} satisfy the recurrence relation
\begin{equation*}
    (n+2)^3 a(n+2) + 3(2n+3)(9n^2+27n+23) a(n+1) + 729 (n+1)^3 a(n) = 0,
\end{equation*}
which can be verified using Zeilberger's algorithm as follows: For $k \in \{0, 1, \dots, n\}$, define
\begin{align*}
    F_1(n, k) &= \left(\frac{(1/3)_k (2/3)_{n-k}}{(1)_k (1)_{n-k}}\right)^2 (-27)^n\\
    F_2(n, k) &= {3k \choose k} {2k \choose k}^2 {n+k \choose n-k} (-27)^{n-k},
\end{align*}
which are hypergeometric terms in both variables $n,k$. We find rational functions $R_i(n,k)$ so that
\begin{equation*}
\begin{split}
    (n+2)^3 F_i(n+2, k) + 3(2n+3)(9n^2+27n+23) F_i(n+1, k) + 729 (n+1)^3 F_i(n, k) \\
    = R_i(n,k+1) F_i(n,k+1) - R_i(n,k) F_i(n,k),
\end{split}
\end{equation*}
and then sum over $k \in \{0, 1, \dots, n\}$. The proof certificates are
\begin{align*}
    R_1(n,k) & = \frac{9k^2(3n-3k+2)^2(3n^2+9n-3k^2+2k+6)}{(n-k+1)^2(n-k+2)^2}, \\
    R_2(n,k) & = -\frac{1458k^4(2n+3)}{(n-k+1)(n-k+2)}.
\end{align*}
\end{proof}

For $\tau \in C_{1/u, \infty}$, we have
\begin{align*}
    \vphantom{1em}_2 F_1 \left( \frac{1}{3}, \frac{1}{3} ; 1; \frac{1}{u} \right)^2 &= (1-1/u)^{1/3} \vphantom{1em}_2 F_1 \left( \frac{1}{3}, \frac{1}{3} ; 1; \frac{1}{u} \right) \vphantom{1em}_2 F_1 \left( \frac{2}{3}, \frac{2}{3} ; 1; \frac{1}{u} \right) \\
    &= (1-1/u)^{1/3} \sum_{n=0}^\infty \sum_{k=0}^n \left(\frac{(1/3)_k (2/3)_{n-k}}{(1)_k (1)_{n-k}}\right)^2 \left(\frac{1}{u}\right)^n. \\
    &= (1-1/u)^{1/3} \sum_{n=0}^\infty \sum_{k=0}^n {3k \choose k} {2k \choose k}^2 {n+k \choose n-k} (-27)^{-k} \left(\frac{1}{u}\right)^n,
\end{align*}
where the first equality comes from Euler's transformation and the third from Lemma~\ref{lem:wz-2}. Using the above identity, Theorem~\ref{thm:u=infty} reduces to
\begin{equation*}
    \sum_{n=0}^\infty \left(-1 + \frac{u+3}{6} - \frac{u^2+18u-27}{6(u-9)} s_2(\tau) + n (u-1)\right) \sum_{k=0}^n {3k \choose k} {2k \choose k}^2 {n+k \choose n-k} (-27)^{-k} \left(\frac{1}{u}\right)^n = \frac{a}{\pi \sqrt{d}} u.
\end{equation*}
The value $\tau = \frac{-1+\sqrt{-3}}{2}$ is such that $u(\tau)$ is rational, is as in \eqref{Chudnovsky}, and lies in $C_{1/u,\infty}$. Thus the above equation holds for this value of $\tau$. We state the identity obtained in Table~\ref{tab:double}.

\subsection{Case 2A}

Setting $a = \frac{1}{8}, b = \frac{3}{8}$ in Clausen's formula \cite[p. 116]{Andrews}, we obtain
\begin{equation*}
    \vphantom{1em}_2 F_1 \left( \frac{1}{8}, \frac{3}{8} ; 1; \frac{1}{v} \right)^2 = \vphantom{1em}_3 F_2 \left( \frac{1}{4}, \frac{3}{4}, \frac{1}{2} ; 1, 1 ; \frac{1}{v} \right).
\end{equation*}
Using the above identity, Theorem~\ref{thm:v=infty} reduces to
\begin{equation*}
\label{eqn:v=infty}
    \sum_{n=0}^{\infty} \left(\frac{1}{6} \left(\frac{t+2}{t+1}\right) - \frac{1}{6}\frac{(t-1)(t+8)}{(t-4)(t+1)} s_2(\tau) + n\right) \frac{\left(\frac{1}{2}\right)_n \left(\frac{1}{4}\right)_n \left(\frac{3}{4}\right)_n}{n!^3} v^{-n}  = \frac{a}{\pi \sqrt{d}} \frac{t-1}{t+1}.
\end{equation*}
The values $\tau = \frac{-1+\sqrt{-5}}{2}, \frac{-1+\sqrt{-7}}{2}, \frac{-1+\sqrt{-9}}{2}, \frac{-1+\sqrt{-13}}{2}, \frac{-1+\sqrt{-25}}{2}, \frac{-1+\sqrt{-37}}{2},
\sqrt{-1},
\frac{\sqrt{-6}}{2}, \frac{\sqrt{-10}}{2}, \frac{\sqrt{-18}}{2}, \frac{\sqrt{-22}}{2}, \frac{\sqrt{-58}}{2}$ are such that $v(\tau)$ is rational, are as in \eqref{Chudnovsky}, and lie in $C_{1/v,\infty}$. Thus the above equation holds for these values of $\tau$. We state all the possible identities in Table~\ref{tab:single}.

\subsection{Case 3A}

Setting $a = \frac{1}{6}, b = \frac{1}{3}$ in Clausen's formula \cite[p. 116]{Andrews}, we obtain
\begin{equation*}
    \vphantom{1em}_2 F_1 \left( \frac{1}{6}, \frac{1}{3} ; 1; \frac{1}{w} \right)^2 = \vphantom{1em}_3 F_2 \left( \frac{1}{3}, \frac{2}{3}, \frac{1}{2} ; 1, 1 ; \frac{1}{w} \right).
\end{equation*}
Using the above identity, Theorem~\ref{thm:w=infty} reduces to
\begin{equation*}
\label{eqn:w=infty}
    \sum_{n=0}^{\infty} \left(\frac{1}{6} \left(\frac{u+3}{u+1}\right) - \frac{1}{6}\frac{u^2+18u-27}{(u-9)(u+1)} s_2(\tau) + n\right) \frac{\left(\frac{1}{2}\right)_n \left(\frac{1}{3}\right)_n \left(\frac{2}{3}\right)_n}{n!^3} w^{-n}  = \frac{a}{\pi \sqrt{d}} \frac{u-1}{u+1}.
\end{equation*}
The values $\tau = \frac{\sqrt{-6}}{3}, \frac{\sqrt{-12}}{3}, \frac{\sqrt{-15}}{3},
\frac{-1+\sqrt{-3}}{2},
\frac{-3+\sqrt{-51}}{6}, \frac{-3+\sqrt{-75}}{6}, \frac{-3+\sqrt{-123}}{6}, \frac{-3+\sqrt{-147}}{6}, \frac{-3+\sqrt{-267}}{6}$ are such that $w(\tau)$ is rational, are as in \eqref{Chudnovsky}, and lie in $C_{1/w,\infty}$. Thus the above equation holds for these values of $\tau$. We state all the possible identities in Table~\ref{tab:single}.

\begin{table}[H]
\centering
\begin{tabular}{ccccccc}
\hline \\[-8pt]
Name & $\tau$ & A & B & s(n) & C & D \\[2pt]
\hline \\[-7pt]
1B & $\sqrt{-1}$ & 1 & 0 & $\displaystyle\sum_{k=0}^n {6k \choose 3k} {3k \choose k} {6n-6k \choose 3n-3k} {3n-3k \choose n-k}$ & 864 & 1 \\[12pt]
\hline \\[-7pt]
2B & $\dfrac{\sqrt{-1}}{2}$ & 1 & -1 & $\displaystyle\sum_{k=0}^n {4k \choose 2k} {2k \choose k} {4n-4k \choose 2n-2k} {2n-2k \choose n-k}$ & 72 & 9 \\[14pt]
& $\dfrac{\sqrt{-2}}{2}$ & 1 & 0 & $\displaystyle \sum_{k=0}^n {4k \choose 2k} {2k \choose k} {4n-4k \choose 2n-2k} {2n-2k \choose n-k}$ & 128 & $\sqrt{2}$ \\[14pt]
& $\sqrt{-1}$ & 8 & 1 & $\displaystyle\sum_{k=0}^n {4k \choose 2k} {2k \choose k} {4n-4k \choose 2n-2k} {2n-2k \choose n-k}$ & 576 & $\dfrac{9}{2}$ \\[14pt]
& $\dfrac{-1+\sqrt{-3}}{2}$ & 4 & 1 & $\displaystyle\sum_{k=0}^n {4k \choose 2k} {2k \choose k} {4n-4k \choose 2n-2k} {2n-2k \choose n-k}$ & -192 & $\sqrt{3}$ \\[14pt]
& $\dfrac{-1+\sqrt{-7}}{2}$ & 8 & 1 & $\displaystyle\sum_{k=0}^n {4k \choose 2k} {2k \choose k} {4n-4k \choose 2n-2k} {2n-2k \choose n-k}$ & -4032 & $\dfrac{9\sqrt{7}}{8}$ \\[12pt]
\hline \\[-7pt]
3B & $\dfrac{\sqrt{-3}}{3}$ & 1 & 0 & $\displaystyle\sum_{k=0}^n {3k \choose k} {2k \choose k} {3n-3k \choose n-k} {2n-2k \choose n-k}$ & 54 & $\sqrt{3}$ \\[14pt]
& $\dfrac{-1+\sqrt{-3}}{2}$ & 9 & 2 & $\displaystyle\sum_{k=0}^n {3k \choose k} {2k \choose k} {3n-3k \choose n-k} {2n-2k \choose n-k}$ & -216 & $\dfrac{8\sqrt{3}}{3}$ \\[12pt]
\hline
\end{tabular}
\caption{Convolutional double summation rational Ramanujan type series}
\label{tab:convolutional-double}
\end{table}

\begin{table}[H]
\centering
\begin{tabular}{ccccccc}
\hline \\[-8pt]
Name & $\tau$ & A & B & s(n) & C & D \\[2pt]
\hline \\[-7pt]
2C & $\sqrt{-1}$ & 1 & 0 & $\displaystyle\sum_{k=0}^n {2k \choose k}^2 {2n-2k \choose n-k}^2$ & 32 & 2 \\[12pt]
\hline
\end{tabular}
\caption{Convolutional/Euler double summation rational Ramanujan type series}
\label{tab:convolutional-euler-double}
\end{table}

\begin{table}[H]
\renewcommand{\arraystretch}{1.95}
\centering
\begin{tabular}{ccccccc}
\hline
Name & $\tau$ & A & B & s(n) & C & D \\
\hline
2B & $\frac{\sqrt{-2}}{2}$ & 4 & 1 & ${2n \choose n}^3$ & -64 & 2 \\
& $\sqrt{-1}$ & 6 & 1 & ${2n \choose n}^3$ & -512 & $2\sqrt{2}$ \\
& $\frac{-1+\sqrt{-3}}{2}$ & 6 & 1 & ${2n \choose n}^3$ & 256 & 4 \\
& $\frac{-1+\sqrt{-7}}{2}$ & 42 & 5 & ${2n \choose n}^3$ & 4096 & 16 \\
\hline
2A & $\frac{-1+\sqrt{-5}}{2}$ & 20 & 3 & ${4n \choose 2n} {2n \choose n}^2$ & -1024 & 8 \\
& $\frac{-1+\sqrt{-7}}{2}$ & 65 & 8 & ${4n \choose 2n} {2n \choose n}^2$ & -3969 & $9\sqrt{7}$ \\
& $\frac{-1+\sqrt{-9}}{2}$ & 28 & 3 & ${4n \choose 2n} {2n \choose n}^2$ & -12288 & $\frac{16\sqrt{3}}{3}$ \\
& $\frac{-1+\sqrt{-13}}{2}$ & 260 & 23 & ${4n \choose 2n} {2n \choose n}^2$ & -82944 & 72 \\
& $\frac{-1+\sqrt{-25}}{2}$ & 644 & 41 & ${4n \choose 2n} {2n \choose n}^2$ & -6635520 & $\frac{288\sqrt{5}}{5}$ \\
& $\frac{-1+\sqrt{-37}}{2}$ & 21460 & 1123 & ${4n \choose 2n} {2n \choose n}^2$ & -199148544 & 3528 \\
& $\sqrt{-1}$ & 7 & 1 & ${4n \choose 2n} {2n \choose n}^2$ & 648 & $\frac{9}{2}$ \\
& $\frac{\sqrt{-6}}{2}$ & 8 & 1 & ${4n \choose 2n} {2n \choose n}^2$ & 2304 & $2\sqrt{3}$ \\
& $\frac{\sqrt{-10}}{2}$ & 10 & 1 & ${4n \choose 2n} {2n \choose n}^2$ & 20736 & $\frac{9\sqrt{2}}{4}$ \\
& $\frac{\sqrt{-18}}{2}$ & 40 & 3 & ${4n \choose 2n} {2n \choose n}^2$ & 614656 & $\frac{49\sqrt{3}}{9}$ \\
& $\frac{\sqrt{-22}}{2}$ & 280 & 19 & ${4n \choose 2n} {2n \choose n}^2$ & 2509056 & $18\sqrt{11}$ \\
& $\frac{\sqrt{-58}}{2}$ & 26390 & 1103 & ${4n \choose 2n} {2n \choose n}^2$ & 24591257856 & $\frac{9801\sqrt{2}}{4}$ \\
\hline
3A & $\frac{\sqrt{-6}}{3}$ & 6 & 1 & ${3n \choose n} {2n \choose n}^2$ & 216 & $3\sqrt{3}$ \\
& $\frac{\sqrt{-12}}{3}$ & 15 & 2 & ${3n \choose n} {2n \choose n}^2$ & 1458 & $\frac{27}{4}$ \\
& $\frac{\sqrt{-15}}{3}$ & 33 & 4 & ${3n \choose n} {2n \choose n}^2$ & 3375 & $\frac{15\sqrt{3}}{2}$ \\
& $\frac{-1+\sqrt{-3}}{2}$ & 5 & 1 & ${3n \choose n} {2n \choose n}^2$ & -192 & $\frac{4\sqrt{3}}{3}$ \\
& $\frac{-3+\sqrt{-51}}{6}$ & 51 & 7 & ${3n \choose n} {2n \choose n}^2$ & -1728 & $12\sqrt{3}$ \\
& $\frac{-3+\sqrt{-75}}{6}$ & 9 & 1 & ${3n \choose n} {2n \choose n}^2$ & -8640 & $\frac{4\sqrt{15}}{5}$ \\
& $\frac{-3+\sqrt{-123}}{6}$ & 615 & 53 & ${3n \choose n} {2n \choose n}^2$ & -110592 & $96\sqrt{3}$ \\
& $\frac{-3+\sqrt{-147}}{6}$ & 165 & 13 & ${3n \choose n} {2n \choose n}^2$ & -326592 & $\frac{108\sqrt{7}}{7}$ \\
& $\frac{-3+\sqrt{-267}}{6}$ & 14151 & 827 & ${3n \choose n} {2n \choose n}^2$ & -27000000 & $1500\sqrt{3}$ \\
\hline
\end{tabular}
\caption{Clausen single summation rational Ramanujan type series}
\label{tab:single}
\end{table}

\begin{table}[H]
\centering
\begin{tabular}{ccccccc}
\hline \\[-8pt]
Name & $\tau$ & A & B & s(n) & C & D \\[2pt]
\hline \\[-7pt]
2B & $\sqrt{-1}$ & 9 & 2 & $\displaystyle\sum_{k=0}^n {4k \choose 2k} {2k \choose k}^2 {n+k \choose n-k} (-64)^{n-k}$ & 512 & $4$ \\[14pt]
& $\dfrac{-1+\sqrt{-3}}{2}$ & 1 & 0 & $\displaystyle\sum_{k=0}^n {4k \choose 2k} {2k \choose k}^2 {n+k \choose n-k} (-64)^{n-k}$ & -256 & $\dfrac{4\sqrt{3}}{9}$ \\[14pt]
& $\dfrac{-1+\sqrt{-7}}{2}$ & 9 & 1 & $\displaystyle\sum_{k=0}^n {4k \choose 2k} {2k \choose k}^2 {n+k \choose n-k} (-64)^{n-k}$ & -4096 & $\dfrac{64\sqrt{7}}{49}$ \\[12pt]
\hline \\[-7pt]
3B & $\dfrac{-1+\sqrt{-3}}{2}$ & 8 & 1 & $\displaystyle\sum_{k=0}^n {3k \choose k} {2k \choose k}^2 {n+k \choose n-k} (-27)^{n-k}$ & -243 & $3\sqrt{3}$ \\[12pt]
\hline
\end{tabular}
\caption{Euler double summation rational Ramanujan type series}
\label{tab:double}
\end{table}

\section{Ramanujan type series and relation to the work of Chan-Cooper and Z.W.\ Sun}

\label{uniqueness}

As mentioned in the introduction, our list of $36$ Clausen cases coincides exactly with the list of single summation formulae for $1/\pi$ in \cite{Chan-Cooper}. The exact correspondence is
\begin{enumerate}
\item Case 1A (Clausen cases only) $\leftrightarrow$ \cite[Table 3, level $\ell = 1$]{Chan-Cooper} (see \cite{Chen})
\item Case 2A (Clausen cases only) $\leftrightarrow$ \cite[Table 4, level $\ell = 2$]{Chan-Cooper}
\item Case 3A (Clausen cases only) $\leftrightarrow$ \cite[Table 5, level $\ell = 3$]{Chan-Cooper}
\item Case 2B (Clausen cases only) $\leftrightarrow$ \cite[Table 6, level $\ell = 4$]{Chan-Cooper}.
\end{enumerate}
This gives evidence that the list of single summation formulae in \cite{Chan-Cooper} is complete: The choice of $\ell$ corresponds to arithmetic triangle groups which give rise to Clausen cases of a Chudnovsky-Ramanujan type formula. The choice of $N$ is justified by requiring a rational value of the uniformizer at an imaginary quadratic irrational. This can be completely enumerated by using known lists of imaginary quadratic orders of class number $\le 2$.

In \cite{Chan-Cooper}, an additional $36$ double summation series for $1/\pi$ for levels 1-4 are derived. Only three of these are rational, given by the first, third, and fourth entries of Table~\ref{tab:double}. The remaining $33$ double summation series are Euler double summation series for the 1B, 2B, 3B, 2C cases (which correspond to the level 1, 2, 3, 4 double summation cases in \cite{Chan-Cooper}, respectively) with quadratic singular values.

Our framework gives rational double summation formulae of a similar form but not found in \cite{Chan-Cooper}. 
For instance, the entry for 2C in Table~\ref{tab:convolutional-euler-double} can be recast as
\begin{equation}
    \sum_{k=0}^\infty (-1)^k \sum_{i=0}^k {{2k-2i}\choose{k-i}}^2  {{2i}\choose{i}}^2 k \left(-\frac{1}{32}\right)^k = \frac{2}{\pi},
\end{equation}
which belongs to the family of formulae corresponding to the case $\ell = 4 \;(y = -1/32, \mu = -1/2)$ of \cite[equation (14)]{Chan-Cooper}. Although \cite{Chan-Cooper} contains four formulae belonging to the same family as the above formula, it doesn't contain the above formula. In fact, none of these four formulae correspond to a rational value of $y$. This series can be found in \cite[(1.1)]{sun2}. Another rational double summation example not found in \cite{Chan-Cooper} is the second entry in Table~\ref{tab:double}, which can be also be found in \cite[(1.15)]{sun2} in a different form. 


Now we discuss the relation to some of the formulae in \cite{sun2}. Formulae \cite[(1.1)-(1.4), (1.9)]{sun2} are convolutional double summation rational Ramanujan series, and appear in our Tables~\ref{tab:convolutional-double} and \ref{tab:convolutional-euler-double}. The remaining 4 entries in our Table~\ref{tab:convolutional-double} do not appear in \cite{sun2} nor \cite{Chan-Cooper}. Formulae \cite[(1.15)-(1.17)]{sun2} are Euler double summation rational Ramanujan series and are listed in Table~\ref{tab:double}.


\end{document}